\documentclass[a4paper,11pt]{amsart}
\usepackage{amssymb,amscd,amsxtra,graphicx,psfrag,xypic}

\usepackage[pdftex]{hyperref}
\hypersetup{%
bookmarksnumbered=true,%
colorlinks=true,%
setpagesize=false,%
pdftitle={},%
pdfauthor={Takuzo Okada}}

\theoremstyle{plain}
\newtheorem{Thm}{Theorem}[section]
\newtheorem{Lem}[Thm]{Lemma}

\newtheorem{Prop}[Thm]{Proposition}

\theoremstyle{definition}
\newtheorem{Def}[Thm]{Definition}
\newtheorem{Def-Lem}[Thm]{Definition-Lemma}
\newtheorem{Cond}[Thm]{Condition}
\newtheorem{Rem}[Thm]{Remark}

\newtheorem*{Ack}{Acknowledgments}

\newcommand{\Aut}{\operatorname{Aut}}
\newcommand{\Bir}{\operatorname{Bir}}

\newcommand{\Proj}{\operatorname{Proj}}
\newcommand{\prt}{\partial}
\newcommand{\Sing}{\operatorname{Sing}}

\newcommand{\Cl}{\operatorname{Cl}}

\newcommand{\Bs}{\operatorname{Bs}}

\newcommand{\mult}{\operatorname{mult}}

\newcommand{\wt}{\operatorname{wt}}

\newcommand{\bcap}{\bigcap\nolimits}
\newcommand{\bsum}{\sum\nolimits}

\newcommand{\mbA}{\mathbb{A}}

\newcommand{\mbC}{\mathbb{C}}

\newcommand{\mbP}{\mathbb{P}}
\newcommand{\mbQ}{\mathbb{Q}}

\newcommand{\mbZ}{\mathbb{Z}}

\newcommand{\mcH}{\mathcal{H}}
\newcommand{\mcI}{\mathcal{I}}

\newcommand{\mcL}{\mathcal{L}}
\newcommand{\mcM}{\mathcal{M}}

\newcommand{\mcO}{\mathcal{O}}

\newcommand{\mfp}{\mathfrak{p}}

\newcommand{\msp}{\mathsf{p}}

\newcommand{\inj}{\hookrightarrow}

\newcommand{\ratmap}{\dashrightarrow}

\title[$\mathbb{Q}$-Fano threefolds with three birational Mori fiber structures]{$\mathbb{Q}$-Fano threefolds with three birational Mori fiber structures}
\author[Takuzo Okada]{Takuzo Okada}
\address{Department of Mathematics, Faculty of Science and Engineering\endgraf
Saga University, Saga 849-8502 Japan}
\email{okada@cc.saga-u.ac.jp}
\date{}

\begin{document}

\maketitle

\begin{abstract}
In this paper we give first examples of $\mathbb{Q}$-Fano threefolds whose birational Mori fiber structures consist of exactly three $\mathbb{Q}$-Fano threefolds.
These examples are constructed as weighted hypersurfaces in a specific weighted projective space.
We also observe that the number of birational Mori fiber structures does not behave upper semi-continuously in a family of $\mathbb{Q}$-Fano threefolds.
\end{abstract}

\section{Introduction} \label{sec:intro}

A Mori fiber space which is birational to a given variety is called a {\it birational Mori fiber structure} of the variety.
We say that a $\mbQ$-Fano variety $X$ with Picard number one is {\it birationally rigid} (resp.\ {\it birationally birigid}) if the birational Mori fiber structures of $X$ consist of a single element $X$ (resp.\ exactly two elements including $X$).
There are many birationally rigid $\mbQ$-Fano varieties such as nonsingular hypersurfaces of degree $n+1$ in $\mbP^{n+1}$ for $n \ge 3$ (\cite{IM, dF}) and quasismooth anticanonically embedded $\mbQ$-Fano threefold weighted hypersurfaces (\cite{CPR, CP}).
Compared to birational rigidity, $\mbQ$-Fano varieties with finite birational Mori fiber structures (or with finite pliability) are less known.
Corti-Mella \cite{CM} proved that a quartic threefold with a specific singular point is birationally birigid.
Cheltsov-Grinenko \cite{CG} constructed an example of a birationally birigid complete intersection of a quadric and a cubic in $\mbP^5$ with a single ordinary double point.
In a series of papers \cite{Okada1,Okada2,Okada3}, we proved that $19$ families and $35$ families of $\mbQ$-Fano threefold weighted complete intersections consist of birationally rigid and birationally birigid $\mbQ$-Fano threefolds, respectively (see also \cite{AZ1}).
There are other interesting examples of birationally non-rigid $\mbQ$-Fano threefolds \cite{AK, AZ2, BZ} but their birational Mori fiber structures are yet to be determined.

The aim of this paper is to construct first examples of $\mbQ$-Fano threefolds with exactly three birational Mori fiber structures.
We also observe that the number of birational Mori fiber structures does not behave upper semi-continuously in a family.
The main objects of this paper are weighted hypersurfaces of degree $8$ in the weighted projective space $\mbP (1,1,2,2,3)$.
We explain known results for this family.

\begin{Thm}[{\cite{CP, CPR}}]
A quasismooth weighted hypersurface of degree $8$ in \linebreak $\mbP (1,1,2,2,3)$ is birationally rigid.
\end{Thm}

\begin{Thm}[{\cite{Okada2}}]
A $\mbQ$-Fano weighted hypersurface of degree $8$ in $\mbP (1,1,2,2,3)$ with a single $cAx/2$ singular point together with some other terminal quotient singular points is birationally birigid.
More precisely, it is birational to a quasismooth $\mbQ$-Fano weighted complete intersection of type $(6,8)$ in $\mbP (1,1,2,3,4,4)$ and it is not birational to any other Mori fiber space.
\end{Thm}

We consider further special members that admit two $cAx/2$ singular points and determine the birational Mori fiber structures of them.
We state the main theorem of this paper.
In the statement, $\mbP (1,1,2,2,3)$ (resp.\ $\mbP (1,1,2,3,4,4)$) is the weighted projective space with homogeneous coordinates $x_0,x_1,y_0,y_1$ and $z$ of degree respectively $1,1,2,2$ and $3$ (resp.\ $x_0,x_1,y,z,s_0$ and $s_1$ of degree respectively $1,1,2,3,4$ and $4$).

\begin{Thm} \label{mainthm}
Let $X'$ be a $\mbQ$-Fano weighted hypersurface
\[
X' = (y_0^2 y_1^2 + y_0 a_6 + y_1 b_6 + c_8 = 0) \subset \mbP (1,1,2,2,3),
\]
where $a_6, b_6, c_8 \in \mbC [x_0,x_1,z]$ are homogeneous polynomials of degree respectively $6$, $6$, $8$.
Then $X'$ is birational to $\mbQ$-Fano weighted complete intersections
\[
X_1 = (s_0 y + s_1 y + a_6 = s_0 s_1 - y b_6 - c_8 = 0) \subset \mbP (1,1,2,3,4,4),
\]
and
\[
X_2 = (s_0 y + s_1 y + b_6 = s_0 s_1 - y a_6 - c_8 = 0) \subset \mbP (1,1,2,3,4,4),
\]
and not birational to any other Mori fiber space.
Moreover we have the following.
\begin{enumerate}
\item If $(a_6,b_6,c_8)$ is asymmetric $($see \emph{Definition \ref{def:sym}}$)$, then $X_1$ is not isomorphic to $X_2$ and the birational Mori fiber structures of $X'$ consist of three $\mbQ$-Fano threefolds $X', X_1$ and $X_2$.
\item If $(a_6,b_6,c_8)$ is symmetric, then $X_1$ is isomorphic to $X_2$ and the birational Mori fiber structures of $X'$ consist of two $\mbQ$-Fano threefolds $X'$ and $X_1 \cong X_2$.
\end{enumerate}
\end{Thm}

In the above theorem, the members $X'$ with the property (1) are more general than those with the property (2).
We observe through the above theorems that the number of birational Mori fiber structures increases as we specialize $\mbQ$-Fano threefolds in a family except for the final specialization from (1) to (2) in Theorem \ref{mainthm} where the number decreases.
Therefore the number of birational Mori fiber structures does not behave upper semi-continuously in a family. 
A similar observation is also given in \cite{CG}. 

\begin{Ack}
The author would like to thank Dr.~Hamid Ahmadinezhad for valuable comments.
He also would like to thank Professor Takashi Kishimoto for warm encouragement.
He is grateful to the referees for numerous constructive suggestions.
The author is partially supported by JSPS Grant-in-Aid for Young Scientists No.~26800019.
\end{Ack}

\section{Maximal and Sarkisov extractions} \label{sec:maxsing}

Notion of maximal singularities for Fano varieties firstly appeared in \cite{IM} and was developed, applied by Iskovskikh, Pukhlikov, Cheltsov, Park, etc (see \cite{Puk} for details).
The recent result of de Fernex \cite{dF} brought a new idea to this subject.
A version of maximal singularity was introduced by Corti in his study \cite{Corti1} of Sarkisov program and  was applied in \cite{CPR}.

We recall the definition of maximal extraction and center which are due to Corti and define a version of them.
Throughout this section, let $X$ be a $\mbQ$-Fano variety with Picard number $1$.
By a $\mbQ$-Fano variety, we mean a normal projective $\mbQ$-factorial variety with only terminal singularities whose anticanonical divisor is ample.
By a {\it divisorial extraction} $\varphi \colon Y \to X$, we mean an extremal divisorial extraction in the Mori category.
We sometimes write $\varphi \colon (E \subset Y) \to (\Gamma \subset X)$, which means that $E$ is the exceptional divisor of the extraction $\varphi$ and $\Gamma = \varphi (E)$ is the center of $\varphi$.

\begin{Def} \label{def:max}
A divisorial extraction $\varphi \colon (E \subset Y) \to (\Gamma \subset X)$ is called a {\it strong maximal extraction} (resp.\ {\it weak maximal extraction}) if there is a movable linear system $\mcH \sim_{\mbQ} - n K_X$ on $X$ such that the inequality and equality
\[
\frac{1}{n} > c (X,\mcH) = \frac{a_E (K_X)}{\mult_E (\mcH)} \quad\left(\text{resp. } \frac{1}{n} > \frac{a_E (K_X)}{\mult_E (\mcH)} \right)
\]
hold, where $a_E (K_X)$ is the discrepancy of $K_X$ along $E$, $\mult_E (\mcH)$ is the multiplicity of $\mcH$ along $E$ and
\[
c (X,\mcH) := \max \{ \lambda \mid K_X + \lambda \mcH \text{ is canonical} \}
\]
is the canonical threshold of the pair $(X,\mcH)$.
The center $\Gamma$ of a strong (resp.\ weak) maximal extraction is called a {\it strong} (resp.\ {\it weak}) {\it maximal center}.
\end{Def}

A strong maximal extraction is called a maximal extraction in \cite{Corti1}.
We emphasize that weak maximal center is also defined in \cite{CPR} but the definition given there is different from ours.
A maximal singularity in the original sense (introduced by Iskovskikh and Manin) is an exceptional divisor $E$ (not necessarily an exceptional divisor of a divisorial extraction) over $X$ such that there is a movable linear system $\mcH \sim_{\mbQ} - n K_X$ satisfying $\mult_E (\mcH) > n a_E (K_X)$.
It follows that the exceptional divisor of a weak maximal extraction is a maximal singularity in the original sense.

\begin{Rem}
As far as the author knows, notion of weak maximal extraction has never appeared in the literature (although it is just a weaker version of strong maximal extraction), so there will be no confusion.
However, a weak maximal center is also defined in \cite{CPR} and the definition is different from ours: a weak maximal center in our sense is the center of a weak maximal extraction while a weak maximal center in \cite{CPR} is the center of a maximal singularity in the original sense.
We emphasize that a weak maximal center in this paper is always the one given in Definition \ref{def:max}.
\end{Rem}

\begin{Def}
A Sarkisov link $\sigma \colon V \ratmap V'$ between Mori fiber spaces $V/S$ and $V'/S'$ is a birational map that sits in the commutative diagram
\[
\xymatrix{
W \ar[d]_{\varphi} \ar@{-->}[r]^{\tau} & W' \ar[d]^{\varphi'} \\
V \ar@{-->}[r]_{\sigma} & V'}
\]
where each of $\varphi$ and $\varphi'$ is either an identity or a divisorial extraction and $\tau$ is either an identity or a composite of inverse flips, flops and flips.
In the case where $\varphi$ (resp.\ $\varphi'$) is a divisorial extraction, we say that the link $\sigma$ {\it starts} (resp.\ {\it ends}) with the divisorial extraction $\varphi$ (resp.\ $\varphi'$). 
\end{Def}

Note that, for a $\mbQ$-Fano variety $X$ with Picard number $1$, any Sarkisov link $X \ratmap X'/S'$ to a Mori fiber space $X'/S'$ starts with a divisorial extraction.
Note also that a Sarkisov link starting with a given divisorial extraction $\varphi \colon Y \to X$ is unique if it exists.

\begin{Def}
A divisorial extraction $\varphi \colon Y \to X$ is called a {\it Sarkisov extraction} if there is a Sarkisov link starting with $\varphi$.
The center on $X$ of a Sarkisov extraction is called a {\it Sarkisov center}.
\end{Def}

\begin{Lem} \label{lem:maxweakS}
For a divisorial contraction $\varphi \colon (E \subset Y) \to (\Gamma \subset X)$, we have the following implications.
\begin{enumerate}
\item If $\varphi$ is a strong maximal extraction, then it is a Sarkisov extraction.
\item If $\varphi$ is a Sarkisov extraction, then it is a weak maximal extraction.
\end{enumerate}
\end{Lem} 

\begin{proof}
The assertion (1) follows from \cite{Corti1} (see the proof of (5.4) Theorem therein).
We prove (2).
The following proof may be straightforward for specialists but we include it for readers' convenience.
Suppose that $\varphi$ is a Sarkisov extraction and let $\sigma \colon X \ratmap X'/S'$ be the Sarkisov link starting with $\varphi$.
If $X'$ is a $\mbQ$-Fano variety with Picard number $1$, then we have the following commutative diagram
\[
\xymatrix{
Y \ar[d]_{\varphi} \ar@{-->}[r]^{\tau} & Y' \ar[d]^{\psi} \\
X \ar@{-->}[r]_{\sigma} & X',}
\]
and otherwise we have the commutative diagram
\[
\xymatrix{
Y \ar[d]_{\varphi} \ar@{-->}[rd]^{\tau} & \\
X \ar@{-->}[r]_{\sigma} & X',}
\]
where $\psi \colon (E' \subset Y') \to (\Gamma' \subset X')$ is an extremal divisorial extraction and $\tau$ is a small birational map.
Let $V$ be a nonsingular projective variety that admit birational morphisms $p \colon V \to X$ and $q \colon V \to X'$ such that $q = \sigma \circ p$.
We assume that $p$ factors through $Y$ and that $q$ factors through $Y'$ if $X'$ is a $\mbQ$-Fano variety with Picard number $1$. 
By a slight abuse of notation, we denote by $E$ and $E'$ the proper transforms of $E$ and $E'$ on $V$, respectively.
Then, since $\tau$ is an isomorphism in codimension one, $E$ (respectively, $E'$) is the unique $p$-exceptional divisor (respectively, $q$-exceptional divisor) that is not $q$-exceptional (respectively, $p$-exceptional).
Let $H'$ be a very ample divisor on $X'$ and let $n'$ be the rational number such that $H' \sim_{\mbQ, S'} - n' K_{X'}$.
We set $\mcH' := |H'|$ and let $\mcH$ be the birational transform of $\mcH'$ on $X$.
Let $n$ be the rational number such that $\mcH \sim_{\mbQ} - n K_X$.
By the Noether--Fano--Iskovskikh inequality (see \cite[(4.2) Theorem]{Corti1}), we have $n > n'$.
We have
\[
\begin{split}
K_V + \frac{1}{n} \mcH_V &= q^* \left( K_{X'} + \frac{1}{n} \mcH' \right) + a' E' + G', \\
&= p^* \left( K_X + \frac{1}{n} \mcH \right) + a E + G,
\end{split}
\]
where $a, a' \in \mbQ$, $G$ and $G'$ are both $p$- and $q$-exceptional divisors and $\mcH_V = q^* \mcH'$.
Note that 
\[
K_{X'} + \frac{1}{n} \mcH' \sim_{\mbQ, S'}  \frac{n' - n}{n} (-K_{X'})
\]
is relatively anti-ample over $S'$ and $K_X + \frac{1}{n} \mcH \sim_{\mbQ} 0$.
Take a sufficiently general curve $C' \subset X'$ that is contracted by $X' \to S'$.
We may assume that $C'$ is disjoint from the image of any $q$-exceptional divisor.
We denote by the same symbol $C'$ its inverse image on $V$.
Then we have $(E' \cdot C') = (G' \cdot C') = (G \cdot C') = 0$ and thus
\[
a (E \cdot C') = \left( K_{X'} + \frac{1}{n} \mcH' \cdot C' \right) < 0.
\]
Since $C'$ is not contained in $E$, the above inequality shows that $(E \cdot C') > 0$ and $a < 0$.
Therefore $\varphi$ is a weak maximal extraction.
\end{proof}

In this paper, we employ the definition of weak maximal extraction and center as follows.

\begin{Def}
A weak maximal extraction and a weak maximal center are called a {\it maximal extraction} and a {\it maximal center}, respectively.
\end{Def} 

An advantage of employing this definition is that the exclusion of a divisorial extraction $\varphi$ as a maximal center immediately implies that of $\varphi$ as a Sarkisov extraction by Lemma \ref{lem:maxweakS}, which enables us to classify Sarkisov links between $\mbQ$-Fano varieties with Picard number $1$.

\begin{Rem}
The bad link method introduced in \cite[Section 5.5]{CPR} excludes a divisorial extraction as a Sarkisov extraction (but not necessarily as a weak maximal extraction).
The approach of the recent paper \cite{AZ1} by Ahmadinezhad and Zucconi can be thought of as a generalization of the bad link method.
It is important to mention that, all the exclusion methods appeared in the literature so far, except for the ones based on the bad link methods explained above, exclude extractions and centers not only as strong maximal ones but also as weak maximal ones. 
\end{Rem}

\section{Preliminaries} \label{sec:prelim}

The aim of this section is to study basic properties of the main objects $X'$, $X_1$ and $X_2$.

\subsection{Quasismoothness}

Let $\mbP = \mbP (a_0,\dots,a_n)$ be a weighted projective space with homogeneous coordinates $x_0,\dots,x_n$. 
We assume that $\mbP$ is well-formed, that is, $\gcd (a_0,\dots,\hat{a}_i,\dots,a_n) = 1$ for each $i$, and let $X$ be a closed subvariety of $\mbP$.
For a non-empty subset $I = \{ i_0,\dots,i_k \}$ of $\{0,\dots,n\}$, we define 
\[
\Pi^{\circ}_I =  \left(\bigcap\nolimits_{i \in I} (x_i \ne 0) \right) \cap \left( \bigcap\nolimits_{j \notin I} (x_j = 0) \right) \subset \mbP 
\]
and call it a {\it coordinate stratum} of $\mbP$ with respect to $I$.
For a $(k+1)$-tuple of non-negative integers $m = (m_0,\dots,m_k)$, we write 
\[
x_I^m = x_{i_0}^{m_0} \cdots x_{i_k}^{m_k}.
\]
\begin{Def}
Let $X$ be a closed subscheme of $\mbP$ and $p \colon \mbA^{n+1} \setminus \{0\} \to \mbP$ the natural projection.
We say that $X$ is {\it quasismooth} if the affine cone $C_X \subset \mbA^{n+1}$ of $X$, which is the closure of $p^{-1} (X)$ in $\mbA^{n+1}$, is smooth outside the origin.
For a non-empty subset $I \subset \{0,\dots,n\}$, we say that $X$ is {\it quasismooth along} $\Pi^{\circ}_I$ if $C_X$ is smooth along $p^{-1} (\Pi^{\circ}_I)$.  
\end{Def}

It follows from the definition that a closed subscheme $X \subset \mbP$ is quasismooth if and only if $X$ is quasismooth along $\Pi^{\circ}_I$ for any non-empty subset $I \subset \{0,\dots,n\}$.

\begin{Def}
Let $M$ be a set of monomials of degree $d$.
We denote by $\Lambda (M)$ the linear system on $\mbP$ spanned by elements in $M$.
Let $M_1$ and $M_2$ be sets of monomials of degree respectively $d_1$ and $d_2$.
We define 
\[
\Lambda (M_1, M_2) = \{ X_1 \cap X_2 \subset \mbP \mid X_1 \in \Lambda (M_1), X_2 \in  \Lambda (M_2) \},
\]
which is the family of weighted complete intersections of type $(d_1, d_2)$ defined as the scheme-theoretic intersection of weighted hypersurfaces in $\Lambda (M_1)$ and $\Lambda (M_2)$.
\end{Def}

We re-state the results of \cite{IF00} on quasismoothness of weighted complete intersections in a generalized form.
Although the statements are slightly different from the original ones, proofs are completely parallel.
More precisely, the proofs can be done by replacing complete linear systems of degree $d$, $d_1$, $d_2$ with linear systems $\Lambda (M)$, $\Lambda (M_1)$, $\Lambda (M_2)$, respectively, in the proofs of the corresponding theorems in \cite{IF00}.
A weighted hypersurface of degree $d$ is said to be a {\it linear cone} if  its defining polynomial $f$ can be written as $f = \alpha x_i + (\text{other terms})$ for some $i$ and non-zero $\alpha \in \mbC$. 

\begin{Thm}[{cf.\ \cite[8.1 Theorem]{IF00}}] \label{qsmcriwh}
Let $I = \{i_0,\dots,i_{k-1}\}$ be a non-empty subset of $\{0,\dots,n\}$ and $M$ a set of monomials of degree $d$.
A general weighted hypersurface in $\Lambda (M)$ which is not a linear cone is quasismooth along $\Pi^{\circ}_I$ if and only if one of the following assertions hold.
\begin{enumerate}
\item There exists a monomial $x_I^m = x_{i_0}^{m_0} \cdots x_{i_{k-1}}^{m_{k-1}} \in M$.
\item For $\mu = 1,\dots,k$, there exist monomials
\[
x_I^{m_{\mu}} x_{e_{\mu}} = x_{i_0}^{m_{0,\mu}} \cdots x_{i_{k-1}}^{m_{k-1,\mu}} x_{e_{\mu}} \in M,
\]
where $\{e_{\mu}\}$ are $k$ distinct elements.
\end{enumerate}
\end{Thm}

\begin{Thm}[{cf. \cite[8.7 Theorem]{IF00}}] \label{qsmcriwci}
Let $I = \{i_0,\dots,i_{k-1}\}$ be a non-empty subset of $\{0,\dots,n\}$ and $M_1$, $M_2$ sets of monomials of degree $d_1$, $d_2$, respectively.
A general weighted complete intersection in $\Lambda (M_1, M_2)$ which is not the intersection of a linear cone with another hypersurface is quasismooth along $\Pi^{\circ}_I$ if and only if one of the following assertions hold.
\begin{enumerate}
\item There exist monomials $x_I^{m_1} \in M_1$ and $x_I^{m_2} \in M_2$.
\item There exists a monomial $x_I^m \in M_1$, and for $\mu = 1,\dots,k-1$ there exist monomials $x_I^{m_{\mu}} x_{e_{\mu}} \in M_2$, where $\{e_{\mu}\}$ are $k-1$ distinct elements. 
\item There exists a monomial $x_I^m \in M_2$, and for $\mu = 1,\dots,k-1$ there exist monomials $x_I^{m_{\mu}} x_{e_{\mu}} \in M_1$, where $\{e_{\mu}\}$ are $k-1$ distinct elements. 
\item For $\mu = 1,\dots,k$, there exist monomials $x_I^{m^1_{\mu}} x_{e^1_{\mu}} \in M_1$, and $x_I^{m^2_{\mu}} x_{e^2_{\mu}} \in M_2$, such that $\{e^1_{\mu}\}$ are $k$-distinct elements, $\{e^2_{\mu}\}$ are $k$ distinct elements and $\{e^1_{\mu}, e^2_{\mu}\}$ contains at least $k+1$ distinct elements.
\end{enumerate}
\end{Thm}

Let $\mbP := \mbP (a_0,\dots,a_4)$ be a weighted projective $4$-space with homogeneous coordinates $x_0,x_1,x_2,x_3,x_4$ with $\deg x_i = a_i$ and $V$ a weighted hypersurface in $\mbP$ which contains a weighted complete intersection curve $\Gamma := (x_0 = f = g = 0)$, where $f, g \in \mbC [x_1,x_2,x_3,x_4]$ with $\deg f \le \deg g =: m$.
We give a criterion for quasismoothness of a general member of a suitable linear system on $V$ along $\Gamma$. 
Let $\mcM \subset |\mcO_V (m)|$ be a linear system on $V$ generated by homogeneous polynomials $g, d_1 f,\dots,d_k f,e_1 x_0,\dots,e_l x_0$ of degree $m$, where $k, j$ are some nonnegative integers and $d_i, e_i \in \mbC [x_0,\dots,x_4]$.
In this case, we define $\mcM_f$ and $\mcM_{x_0}$ to be the linear systems spanned by $d_1,\dots,d_k$ and $e_1,\dots,e_l$, respectively.
We define
\[
\operatorname{NQsm} (V) = p (\Sing C_V \setminus \{0\}),
\]
where $C_V \subset \mbA^5$ is the affine cone of $V$ and $p \colon \mbA^5 \setminus \{0\} \to \mbP$ the natural projection, and call it the {\it non-quasismooth locus} of $V$. 

\begin{Prop} \label{prop:qsmcurve}
Let $V \subset \mbP = \mbP (a_0,\dots,a_4)$, $\Gamma = (x_0 = f = g = 0) \subset V$, $\deg f \le \deg g =:m$ and $\mcM \subset \left|\mcO_V (m)\right|$ be as above.
Suppose that $\Gamma$ is quasismooth and that $\Bs \mcM_f \not\supset \Gamma$.
Then a general member of $\mcM$ is quasismooth along $\Gamma \setminus (\operatorname{NQsm} (V) \cup \Bs \mcM_{x_0})$.
\end{Prop}

\begin{proof}
A defining polynomial of $V$ can be written as $b f + c g + x_0 h$ for some $b,c \in \mbC [x_1,\dots,x_4]$ and $h \in \mbC [x_0,\dots,x_4]$.
Let $S \in \mcM$ be a general member.
A section $s$ which cuts out $S$ on $V$ can be written as $s = d f + \alpha g + x_0 e$ for some $\alpha \in \mbC$, $d \in \mbC [x_1,\dots,x_4]$ and $e \in \mbC [x_0,\dots,x_4]$ such that $H_d := (d = 0) \cap X \in \mcM_f$ and $H_e := (e = 0) \cap X \in \mcM_{x_0}$.
Note that $\alpha \ne 0$ and $H_d \not\supset \Gamma$ since $S$ is general and $\Bs \mcM_f \not\supset \Gamma$.
If $\mcM_{x_0} \supset \Gamma$, then the assertion follows immediately (in the sense that the conclusion is vacuous).
Hence we may assume that $\Bs \mcM_{x_0} \not\supset \Gamma$ and $H_e \not\supset \Gamma$.
The restriction to $\Gamma$ of the Jacobian matrix of the affine cone $C_S$ of $S$ can be computed as
\[
J_{C_S} |_{\Gamma} =
\begin{pmatrix}
h & b \frac{\prt f}{\prt x_1} + c \frac{\prt g}{\prt x_1} & b \frac{\prt f}{\prt x_2} + c \frac{\prt g}{\prt x_2} & b \frac{\prt f}{\prt x_3} + c \frac{\prt g}{\prt x_3} & b \frac{\prt f}{\prt x_4} + c \frac{\prt g}{\prt x_4} \\[2mm]
e & d \frac{\prt f}{\prt x_1} + \alpha \frac{\prt g}{\prt x_1} & d \frac{\prt f}{\prt x_2} + \alpha \frac{\prt g}{\prt x_2} & d \frac{\prt f}{\prt x_3} + \alpha \frac{\prt g}{\prt x_3} & d \frac{\prt f}{\prt x_4} + \alpha \frac{\prt g}{\prt x_4}
\end{pmatrix}.
\]
Note that the matrix
\[
\begin{split}
& \begin{pmatrix}
b \frac{\prt f}{\prt x_1} + c \frac{\prt g}{\prt x_1} & b \frac{\prt f}{\prt x_2} + c \frac{\prt g}{\prt x_2} & b \frac{\prt f}{\prt x_3} + c \frac{\prt g}{\prt x_3} & b \frac{\prt f}{\prt x_4} + c \frac{\prt g}{\prt x_4} \\[2mm]
d \frac{\prt f}{\prt x_1} + \alpha \frac{\prt g}{\prt x_1} & d \frac{\prt f}{\prt x_2} + \alpha \frac{\prt g}{\prt x_2} & d \frac{\prt f}{\prt x_3} + \alpha \frac{\prt g}{\prt x_3} & d \frac{\prt f}{\prt x_4} + \alpha \frac{\prt g}{\prt x_4}
\end{pmatrix} \\
&= 
\begin{pmatrix}
b & c \\
d & \alpha
\end{pmatrix}
\begin{pmatrix}
\frac{\prt f}{\prt x_1} & \frac{\prt f}{\prt x_2} & \frac{\prt f}{\prt x_3} & \frac{\prt f}{\prt x_4} \\[2mm]
\frac{\prt g}{\prt x_1} & \frac{\prt g}{\prt x_2} & \frac{\prt g}{\prt x_3} & \frac{\prt g}{\prt x_4}
\end{pmatrix}.
\end{split}
\]
is of rank $2$ at any point of $\Gamma \setminus (\alpha b - c d = 0)$  and is of rank $1$ at any point of $\Gamma \cap (\alpha b - c d = 0)$ since $\Gamma$ is quasismooth and $\alpha \ne 0$.
It follows that $S$ is quasismooth along $\Gamma \setminus (\alpha b - c d = 0)$.
We shall show that $J_{C_S} |_{\Gamma}$ is of rank $2$ at any point $\msp \in \Gamma \setminus (\operatorname{NQsm} (V) \cup \Bs \mcM_{x_0})$.

Assume that $(b = c = 0) \cap \Gamma = \Gamma$, that is,  both $b$ and $c$ vanish along $\Gamma$.
Then $h$ does not vanish at $\msp$ since $V$ is quasismooth at $\msp$.
It follows that $J_{C_S} |_{\Gamma}$ is of rank $2$ at $\msp$.

In the following, we assume that $(b=c=0) \cap \Gamma \ne \Gamma$.
We claim that $(\alpha b - c d = 0) \cap \Gamma$ is a finite set of points.
If $(b=0) \not\supset \Gamma$, then $(\alpha b - c d = 0) \cap \Gamma \ne \Gamma$ for a general choice of $\alpha$ and $d$.
Assume that $(b= 0) \supset \Gamma$.
Then $(c = 0) \not\supset \Gamma$ since $(b=c=0) \cap \Gamma \ne \Gamma$.
In this case $(\alpha b - c d = 0) \cap \Gamma = (c d = 0) \cap \Gamma$ and it is a finite set of points since $H_d \not\supset \Gamma$.

If $\msp \notin (\alpha b - c d = 0)$, then $J_{C_S} |_{\Gamma}$ is of rank $2$ at $\msp$ by the above argument.
It remains to consider the case $\msp \in (\alpha b - c d = 0) \cap \Gamma$.
Since $V$ is quasismooth at $\msp$, the first row of $(J_{C_S}|_{\Gamma}) (\msp)$ is non-zero.
If the entries of the first row of $(J_{C_S}|_{\Gamma}) (\msp)$ are zero except for $h (\msp)$, then $J_{C_S}|_{\Gamma}$ is of rank $2$.
Otherwise there is a non-zero entry in the first row of $(J_{C_S}|_{\Gamma}) (\msp)$ other than $h (\msp)$ and we can choose a general $e$ so that $J_{C_S}|_{\Gamma}$ is of rank $2$ at $\msp$ since $H_e \in \mcM_{x_0}$ and $\msp \notin \Bs \mcM_{x_0}$. 
Since there are only finitely many points in $\Gamma \cap (\alpha b - c d = 0)$, we can choose a general $e$ so that $J_{C_S}|_{\Gamma}$ is of rank $2$ at every point of $\Gamma \cap (\alpha b - c d = 0)$.
This completes the proof.
\end{proof}

\subsection{Generality conditions and their consequences}

In the rest of this paper, the coordinates $x_0, x_1, y_0, y_1, y, z, s_0$ and $s_1$ are of degree $1,1,2,2,2,3,4$ and $4$, respectively.
We set
\[
\mbP (1,1,2,2,3) = \Proj \mbC [x_0,x_1,y_0,y_1,z]
\]
and
\[
\mbP (1,1,2,3,4,4) = \Proj \mbC [x_0,x_1,y,z,s_0,s_1].
\]
Let $a_6, b_6$ and $c_8$ be homogeneous polynomials of degree $6$, $6$ and $8$, respectively, in variables $x_0,x_1,z$.
We define weighted hypersurface
\[
X' = (y_0^2 y_1^2 + y_0 a_6 + y_1 b_6 + c_8 = 0) \subset \mbP (1,1,2,2,3)
\]
and weighted complete intersections
\[
X_1 = (s_0 y + s_1 y + a_6 = s_0 s_1 - y b_6 - c_8 = 0) \subset \mbP (1,1,2,3,4,4),
\]
\[
X_2 = (s_0 y + s_1 y + b_6 = s_0 s_1 - y a_6 - c_8 = 0) \subset \mbP (1,1,2,3,4,4).
\]
We define points of $X'$ as
\[
\msp'_1 = (0 \!:\! 0 \!:\! 1 \!:\! 0 \!:\! 0), \ 
\msp'_2 = (0 \!:\! 0 \!:\! 0 \!:\! 1 \!:\! 0), \ 
\msp'_3 = (0 \!:\! 0 \!:\! 0 \!:\! 0 \!:\! 1),
\]
and points of $X_i$, $i = 1,2$, as
\[
\msp_1 = (0 \!:\! 0 \!:\! 0 \!:\! 0 \!:\! 1 \!:\! 0), \ 
\msp_2 = (0 \!:\! 0 \!:\! 0 \!:\! 0 \!:\! 0 \!:\! 1), \ 
\msp_3 = (0 \!:\! 0 \!:\! 1 \!:\! 0 \!:\! 0 \!:\! 0).
\]

We recall the definition of singularity of type $cAx/2$ and after that we introduce conditions on the triplet $(a_6,b_6,c_8)$.
In the following, $\mbA^4_{x,y,z,u}/\mbZ_2 (a,b,c,d)$ is the quotient of the affine $4$-space with affine coordinates $x,y,z,u$ under the $\mbZ_2 (= \mbZ/2 \mbZ)$-action given by 
\[
(x,y,z,u) \mapsto ((-1)^a x, (-1)^b y, (-1)^c z, (-1)^d u),
\]
and $(g (x,y,z,u) = 0) / \mbZ_2 (a,b,c,d)$ is the quotient of the hypersurface $g = 0$ in $\mbA^4$ for a $\mbZ_2$-invariant polynomial $g$.

\begin{Def}
Let $X$ be a germ of a $3$-dimensional terminal singularity.
We say that the singularity is of type $cAx/2$ if there is an isomorphism
\[
X \cong (x^2 + y^2 + g (z,u) = 0) / \mbZ_2 (0,1,1,1),
\]
where $g (z,u) \in (z,u)^4 \mbC \{z,u\}$ is $\mbZ_2$-invariant.
\end{Def}

\begin{Lem} \label{lem:cAx2verif}
Let 
\[
o \in (x^2 + y^2 + g (z,u) = 0)/\mbZ_2 (0,1,1,1)
\]
be a germ, where $g (z,u)$ is $\mbZ_2$-invariant, and let $f$ be the lowest degree part of $g$.
If $\deg f = 6$ and $f$ does not have a multiple component, then the germ is a $cAx/2$ singular point and there exists a unique divisorial extraction centered at $o$.
\end{Lem}

\begin{proof}
We set $V = (x^2 + y^2 + g (z,u) = 0)/\mbZ_2 (0,1,1,1)$.
We need to show that $o \in V$ is terminal.
Let $\varphi \colon W \to V$ be the weighted blowup of $V$ at $o$ with $\wt (x,y,z,u) = \frac{1}{2} (4,3,1,1)$.
The exceptional divisor $E$ is isomorphic to
\[
(y^2 + f (z,u) = 0) \subset \mbP (4,3,1,1),
\]
where $x,y,z,u$ are thought of as homogeneous coordinates of degree $3,4,1,1$.
We see that $E$ is irreducible and it is straightforward to see that $W$ has a singularity of type $\frac{1}{4} (1,1,3)$ at $(1\!:\!0\!:\!0\!:\!0) \in E$ and is nonsingular elsewhere.
Moreover $K_W = \varphi^*K_V + \frac{1}{2} E$.
This shows that $\varphi$ is a divisorial contraction from a terminal threefold $W$. 
Therefore $o \in V$ is a terminal singularity. 

According to the classification \cite{Ha99,Ka05} of divisorial extractions, there is a unique divisorial extraction centered at $o$ if the lowest degree part of $g$, which is $f$, is not a square (see also \cite[Section 2.2]{Okada2}).
Therefore the proof is completed.
\end{proof}

\begin{Cond} \label{cond}
\begin{enumerate}
\item $X'$ is quasismooth outside the points $\msp'_1$ and $\msp'_2$.
\item The singularities of $X'$ at $\msp'_1$ and $\msp'_2$ are both of type $cAx/2$.
\item Both $X_1$ and $X_2$ are quasismooth outside the point $\msp_3$.
\item The singularities of $X_1$ and $X_2$ at $\msp_3$ are both of type $cAx/2$.
\end{enumerate}
\end{Cond}

\begin{Def} \label{def:M}
For a positive integer $d$ and a polynomial $g$ in variable $y$ or in variables $y_0,y_1$, we define 
\[
M_d = \{\, x_0^k x_1^l z^m \mid k,l,m \ge 0 \text{ and } k + l + 3 m = d \, \}
\]
and
\[
g M_d = \{\, g h \mid h \in M_d \, \}.
\]
\end{Def}

\begin{Prop} \label{prop:verifycond}
\emph{Condition \ref{cond}} is satisfied for a general triplet $(a_6,b_6,c_8)$.
\end{Prop}

\begin{proof}
We set
\[
\begin{split}
N' &= \{ y_0^2 y_1^2 \} \cup y_0 M_6 \cup y_1 M_6 \cup M_8, \\
N_6 &= \{s_0 y, s_1 y\} \cup M_6, \\
N_8 &= \{s_0 s_1\} \cup y M_6 \cup M_8.
\end{split}
\] 
To verify conditions (1) and (3), it is enough to show that general members of $\Lambda (N')$ and $\Lambda (N_6,N_8)$ are quasismooth outside $\msp'_1, \msp'_2$ and $\msp_3$, respectively.
This follows from Theorems \ref{qsmcriwh} and \ref{qsmcriwci}.

Note that $(a_6 = 0) \subset \mbP (1,1,3)$ is quasismooth for a general $a_6$.
We claim that if $(a_6 = 0) \subset \mbP (1,1,3)$ is quasismooth, then $\msp'_1 \in X'$ is of type $cAx/2$.
We work on the open subset where $y_0 \ne 0$.
Then, by setting $y_0 = 1$, $X'$ is defined as
\[
(y_1^2 + a_6 + y_1 b_6 + c_8 = 0) \subset \mbA^4_{x_0,x_1,y_1,z}/\mbZ_2 (1,1,0,1).
\]
Since $(a_6 = 0) \subset \mbP (1,1,3)$ is quasismooth, $z^2 \in a_6$, and hence we may write $a_6 = z^2 + f_6 (x_0,x_1)$ for some $f_6$ after replacing $z$.
It follows again from quasismoothness of $(a_6 = 0) \subset \mbP (1,1,3)$ that $f_6$ does not have a multiple component.
By a suitable analytic coordinate change, the germ $(X',\msp'_1)$ is analytically equivalent to the origin of
\[
(y_1^2 + z^2 + g (x_0,x_1) = 0) \subset \mbA^4_{x_0,x_1,y_1,z} / \mbZ_2 (1,1,0,1),
\]
where the lowest weight term of $g$ is $f_6$. 
By Lemma \ref{lem:cAx2verif}, $\msp'_1$ is of type $cAx/2$.
By the symmetric argument, the point $\msp'_2 \in X'$ is of type $cAx/2$ if $(b_6 = 0) \subset \mbP (1,1,3)$ is quasismooth, and the condition (2) is verified.

We claim that the singularity of $X_2$ at $\msp_3$ is equivalent to that of $X'$ at $\msp'_1$.
By setting $y = 1$ in the defining polynomials of $X_2$, we see that $(X_2, \msp_3)$ is isomorphic to
\[
\begin{split}
(s_0 + & s_1 + b_6 = s_0 s_1 - a_6 - c_8 = 0) 
\subset \mbA^5_{x_0,x_1,z,s_0,s_1}/\mbZ_2 (1,1,1,0,0) \\
& \cong (s_0^2 + a_6 + s_0 b_6 + c_8 = 0) \subset \mbA^4_{x_0,x_1,s_0,z} /\mbZ_2 (1,1,0,1).
\end{split}
\]
Hence the germ $(X_2, \msp_3)$ is isomorphic to $(X', \msp'_1)$.
We have $(X_1,\msp_3) \cong (X',\msp'_2)$ by symmetry.
Therefore the condition (4) follows from (2).
This completes the proof.
\end{proof}

In the following we assume that $(a_6,b_6,c_8)$ satisfies Condition \ref{cond}.
We see that $\Sing (X') = \{\msp'_1,\msp'_2,\msp'_3\}$ and the singularity of $X'$ at $\msp'_3$ is of type $\frac{1}{3} (1,1,2)$, and $\Sing (X_i) = \{\msp_1,\msp_2,\msp_3\}$ and the singularity of $X_i$ at $\msp_1, \msp_2$ are of type $\frac{1}{4} (1,1,3)$. 

\begin{Lem} \label{lem:uniqueextr}
The following assertions hold.
\begin{enumerate}
\item The weighted hypersurfaces
\[
(a_6 = 0) \subset \mbP (1,1,3) \text{ and } (b_6 = 0) \subset \mbP (1,1,3)
\]
are quasismooth.
\item Let $X$ be one of $X'$, $X_1$ and $X_2$, and $\msp$ a singular point of $X$.
Then there is a unique divisorial extraction centered at $\msp$.
\end{enumerate}
\end{Lem}

\begin{proof}
Assume that $C := (a_6 = 0) \subset \mbP (1,1,3)$ is not quasismooth at a point $(\xi_0 \!:\! \xi_1 \!:\! \zeta) \in C$.
Let $\sigma$ be a complex number such that $\sigma^2 =  -c_8 (\xi_0,\xi_1,\zeta)$ and set $\msp = (\xi_0 \!:\! \xi_1 \!:\! 0 \!:\! \zeta \!:\! \sigma \!:\! -\sigma)$.
We see that $\msp \in X_1$ and $X_1$ is not quasismooth at $\msp$.
This is a contradiction because $X_1$ is quasismooth except at $\msp_3$.
Thus $C$ is quasismooth.
Quasismoothness of $(b_6 = 0) \subset \mbP (1,1,3)$ can be proved in the same way using $X_2$.
This shows (1).

The uniqueness of divisorial extraction centered at a terminal quotient singular point follows from \cite{Kaw96}.
We consider $cAx/2$ points.
By the proof of Proposition \ref{prop:verifycond}, after replacing $z$ so that $a_6 = z^2 + f_6 (x_0,x_1)$, the singularity of $X'$ at $\msp'_1$ is equivalent to 
\[
(y_1^2 + z^2 + g (x_0,x_1) = 0) \subset \mbA^4_{x_0,x_1,y_1,z} / \mbZ_2 (1,1,0,1),
\]
where the lowest degree part of $g$ is $f_6$.
By (1), the polynomial $f_6$ does not have a multiple component.
Thus the uniqueness follows from Lemma \ref{lem:cAx2verif}.
(2) follows for $(X', \msp'_2)$ by symmetry and for $(X_1,\msp_3)$ and $(X_2,\msp_3)$ since the singularities of $X_1$ at $\msp_3$ and of $X_2$ at $\msp_3$ are equivalent to those of $X'$ at $\msp'_2$ and at $\msp'_1$, respectively.
This proves (2).
\end{proof}

\begin{Prop}
The varieties $X'$, $X_1$ and $X_2$ are $\mbQ$-factorial.
\end{Prop}

\begin{proof}
This follows from Lemma \ref{QfaccAx} below. 
\end{proof}

\begin{Lem} \label{QfaccAx}
A singular point of type $cAx/2$ is $($analytically$)$ $\mbQ$-factorial.
\end{Lem}

\begin{proof}
Let $(X,o)$ be a germ of singularity of type $cAx/2$.
Then $X$ is analytically equivalent to
\[
(x^2 + y^2 + g (z,t) = 0) \subset \mbA^4 / \mbZ_2 (0,1,1,1),
\] 
where $g (z,t) \in (z,t)^4$ is semi-invariant.
We define 
\[
B = \mbC [[x,y,z,t]]/(x^2 + y^2 + g (z,t))
\] 
and consider the $\mbZ_2$ action of type $(0,1,1,1)$ on $B$.
We see that the completion $\hat{\mcO}_{X,o}$ is isomorphic to $A := B^{\mbZ_2}$.
Since $o \in X$ is an isolated singularity, there is no multiple in the irreducible decomposition $g = g_1 g_2 \cdots g_d$.
We see that 
\[
\Cl (B) = \bigoplus\nolimits_{i=1}^d \mbZ \! \cdot \! [\mfp_i] / \bsum_{i=1}^d [\mfp_i],
\] 
where $\mfp_i = (x - \sqrt{-1} y, g_i)$ is a height $1$ prime ideal of $B$.
Let $j \colon \Cl (A) \to \Cl (B)$ be the homomorphism induced by the injection $A \inj B$.
The image of $j$ is contained in $\Cl (B)^{\mbZ_2}$ and the kernel of $j$ is contained in $H^1 (\mbZ_2, B^*)$ (cf.\ \cite[Theorem16.1]{Fossum}).
The $\mbZ_2$ action on $\Cl (B)$ is given by $[\mfp_i] \mapsto - [\mfp_i]$.
It is easy to see that $\Cl (B)^{\mbZ_2} = 0$ and that $H^1 (\mbZ_2,B^*)$ consists of $2$-torsions.
It follows that $\Cl (A)$ consists of $2$-torsions and in particular we have $\Cl (A) \otimes_{\mbZ} \mbQ = 0$.
This shows that $(X,o)$ is $\mbQ$-factorial.
\end{proof}

\subsection{Condition for $X_1$ and $X_2$ being isomorphic}

We consider a condition on $(a_6,b_6,c_8)$ for $X_1$ and $X_2$ being isomorphic to each other.

\begin{Def} \label{def:sym}
We say that a triplet $(a_6,b_6,c_8)$ is {\it symmetric} if there are non-zero complex numbers $\alpha,\beta,\gamma$ and an automorphism $\tau$ of $\mbP (1,1,3)$ such that $\gamma^3 = \alpha^2 \beta^2$, $\tau^* a_6 = \alpha b_6$, $\tau^*b_6 = \beta a_6$ and $\tau^*c_8 = \gamma c_8$.
A triplet $(a_6,b_6,c_8)$ is called {\it asymmetric} if it is not symmetric.
\end{Def}

\begin{Lem} \label{lem:h_0}
Set $\mbP := \mbP (1,1,2,3,4,4)$.
The homomorphism $H^0 (\mbP, \mcO_{\mbP} (m)) \to H^0 (X_1, \mcO_{X_1} (m))$ is an isomorphism for $m \le 5$.
\end{Lem}

\begin{proof}
We set $X := X_1$ and let $Y = (s_0 s_1 - y b_6 - c_8 = 0) \subset \mbP$ be the weighted hypersurface containing $X$.
Let $S$ be the non-quasismooth locus of $Y$.
We have $\dim S \le 1$ since $X$ is quasismooth outside a single point.
Let $T$ be the union of $S$ and the singular locus of $\mbP$, and we see that $U := \mbP \setminus T$, $Y_U := Y \cap U$ and $X_U := X \cap U$ are nonsingular.
Moreover the codimension of $X \setminus X_U$ in $X$ is at least $2$ since $\dim S \le 1$ and $T \cap X$ is a finite set of points.
Since the codimension in $\mbP$ of each component of $S$ is greater than or equal to $3$, we have $H^i (U, \mcO_U (m)) = H^i (\mbP, \mcO_{\mbP} (m))$ for $i = 0,1,2$ and for any $m$.
This follows by considering the long exact sequence of local cohomologies.
In particular,  we have $H^1 (U, \mcO_U (m)) = H^2 (U, \mcO_U (m)) = 0$ for any $m$.
By the long exact sequence associated to the exact sequence
\[
0 \to \mcO_U (m - 8) \to \mcO_U (m) \to \mcO_{Y_U} (m) \to 0,
\]
we have $H^0 (U, \mcO_U (m)) \cong H^0 (Y_U, \mcO_{Y_U} (m))$ for $m < 8$ and $H^1 (Y_U, \mcO_{Y_U} (m)) = 0$ for any $m$.
Then, by the long exact sequence associated to the exact sequence
\[
0 \to \mcO_{Y_U} (m - 6) \to \mcO_{Y_U} (m) \to \mcO_{X_U} (m) \to 0,
\] 
we have $H^0 (Y_U, \mcO_{Y_U} (m)) \cong H^0 (X_U, \mcO_{X_U} (m))$ for $m < 6$.
This shows that the restriction $H^0 (U, \mcO_U (m)) \to H^0 (X_U, \mcO_{X_U} (m))$ is an isomorphism for $m < 6$.
\end{proof}

\begin{Prop} \label{prop:sym}
$X_1$ is isomorphic to $X_2$ if and only if $(a_6,b_6,c_8)$ is symmetric.
\end{Prop}

\begin{proof}
Assume that there is an isomorphism $\sigma \colon X_1 \to X_2$.
We have $\sigma^* \mcO_{X_2} (m) \cong \mcO_{X_1} (m)$ for any $m$ since $\sigma^* K_{X_2} = K_{X_1}$.
By Lemma \ref{lem:h_0}, the sections $\sigma^*x_0$, $\sigma^*x_1$, $\sigma^*y$, $\sigma^*z$, $\sigma^*s_0$, $\sigma^*s_1$ can be identified with homogeneous polynomials of degree respectively $1,1,2,3,4,4$, and let $\varphi$ be the automorphism of $\mbP (1,1,2,3,4,4)$ inducing $\sigma$.
The divisor which is cut out on $X_1$ by $\sigma^*s_i$ ($i = 0,1$) passes through a singular point of type $\frac{1}{4} (1,1,3)$.
By replacing $\sigma$ with the composite of $\sigma$ and the automorphism of $X_1$ interchanging $s_0$ and $s_1$, we can assume that $\sigma^*s_0$ (resp.\ $\sigma^*s_1$) vanishes at $\msp_2$ (resp.\ $\msp_1$) and does not vanish at $\msp_1$ (resp.\ $\msp_2$).
We may write $\varphi^* s_i = \lambda_i s_i + \lambda'_i y^2 + y q^{(i)} + f^{(i)}$, $\varphi^* z = \nu z + y \ell + g$ and $\varphi^* y = \mu y + h$, where $\lambda_i, \lambda'_i, \mu, \nu \in \mbC$, $q^{(i)}, \ell, g, h \in \mbC [x_0,x_1]$ and $f^{(i)} \in \mbC [x_0,x_1,z]$.
Since the zero loci of $\varphi^* (s_0 y + s_i y + b_6)$ and $\varphi^* (s_0 s_1 - y a_6 - c_8)$ contain $X_1$, we have 
\begin{equation} \label{eq1}
\varphi^* (s_0 y + s_1 y + b_6) = \delta (s_0 y + s_1 y + a_6)
\end{equation}
and
\begin{equation} \label{eq2}
\varphi^* (s_0 s_1 + y a_6 + c_8) = \varepsilon (s_0 s_1 - y b_6 - c_8) + q (s_0 y + s_1 y + a_6)
\end{equation}  
for some non-zero $\delta, \varepsilon \in \mbC$ and $q \in \mbC [x_0,x_1,y]$.
By comparing the terms involving $s_i$ in \eqref{eq1}, we have $\lambda_0 = \lambda_1$, $\mu \ne 0$ and $h = 0$.
We put $\lambda := \lambda_0 = \lambda_1$.
Note that there is no monomial divisible by $y^3$ in $\varphi^*a_6$, $\varphi^*b_6$ and $\varphi^*c_8$.
By comparing terms involving $s_i$ in \eqref{eq2}, we have $\varepsilon = \lambda^2$, $\lambda'_i = 0$, $f^{(i)} = 0$ and $q = \lambda q^{(0)} = \lambda q^{(1)}$.
By comparing terms involving $y^3$ in \eqref{eq2}, we have $\ell = 0$.
It follows that $\varphi^*a_6, \varphi^*b_6, \varphi^*c_8 \in \mbC [x_0,x_1,z]$.
Thus, by comparing terms divisible by $y^2$ in \eqref{eq1}, we have $q^{(0)} = q^{(1)} = 0$.
Therefore, we have $\varphi^* s_i = \lambda s_i$ and $\varphi^* y = \mu y$, $\varphi^*z = \nu z + g (x_0,x_1)$ and $\varphi^*x_i \in \mbC [x_0,x_1]$, and the relations $\varphi^* b_6 = \lambda \mu a_6$, $\mu \varphi^*a_6 = \lambda^2 b_6$ and $\varphi^*c_8 = \lambda^2 c_8$ are satisfied.
By setting $\alpha = \lambda^2/\mu$, $\beta = \lambda \mu$ and $\gamma = \lambda^2$, we observe $\gamma^3 = \alpha^2 \beta^2$.
Thus $(a_6,b_6,c_8)$ is symmetric.

Conversely, if we are given an automorphism $\tau$ of $\mbP (1,1,3)$ and $\alpha,\beta,\gamma \in \mbC$ such that $\gamma^3 = \alpha^2 \beta^2$, $\tau^*a_6 = \alpha b_6$, $\tau^* b_6 = \beta a_6$ and $\tau^* c_8 = \gamma c_8$, then the automorphism $\varphi$ of $\mbP (1,1,2,3,4,4)$ defined by $\varphi^*x_i = \tau^*x_i$ for $i = 0,1$, $\varphi^* z = \tau^* z$ and
\[
\varphi^* y = \frac{\gamma}{\alpha} y, \ \varphi^* s_0 = \frac{\alpha \beta}{\gamma} s_0, \  \varphi^* s_1 = \frac{\alpha \beta}{\gamma} s_1,
\]
restricts to an isomorphism between $X_1$ and $X_2$.
This completes the proof.
\end{proof}

We show that there does exist a symmetric triplet $(a_6,b_6,c_8)$ that satisfies Condition \ref{cond}.

\begin{Prop}
Let $a_6$ and $c_8$ are general homogeneous polynomials in variables $x_0,x_1,z$.
Then the triplet $(a_6,a_6,c_8)$ is symmetric and satisfies \emph{Condition \ref{cond}}.
\end{Prop}

\begin{proof}
Let $X'$ be the weighted hypersurface
\[
X' = (y_0^2 y_1^2 + y_0 a_6 + y_1 a_6 + c_8 = 0) \subset \mbP (1,1,2,2,3)
\]
and let $\Lambda$ be the linear system spanned by $y_0^2 y_1^2$,  $M_8$ and $(y_0 + y_1) M_6$.
A general member $X'$ of $\Lambda$ is quasismooth outside the base locus of $\Lambda$ by the Bertini theorem and the base locus of $\Lambda$ is the set $\{ \msp'_1, \msp'_2,  \msp'_3\}$.
The check of singularity types of $X'$ at $\msp'_1, \msp'_2$ and $\msp'_3$ can be done as in the proof of Proposition \ref{prop:verifycond}.

Let $a_6$ and $c_8$ be general so that $X'$ is quasismooth outside $\{\msp'_1,\msp'_2\}$ and the singularity of $X'$ at $\msp'_1$ and $\msp'_2$ are both of type $cAx/2$.
Let $X$ be the weighted complete intersection
\[
X = (s_0 y + s_1 y + a_6 = s_0 s_1 - y a_6 - c_8 = 0) \subset \mbP (1,1,2,3,4,4).
\]
We have $X = X_1 = X_2$.
It is easy to check that the singularities of $X$ at $\msp_1, \msp_2$ are both of type $\frac{1}{4} (1,1,3)$.
As in the proof of Proposition \ref{prop:verifycond}, we have the equivalences of singularities $(X, \msp_3) \cong (X',\msp'_1)$, hence the singularity of $X$ at $\msp_3$ is of type $cAx/2$.
It remains to show that $X^{\circ} := X \setminus \{\msp_1,\msp_2,\msp_3\}$ is nonsingular.
Instead of proving quasismoothness of $X$ directly, we derive it from the description of singularities of $X'$ by making use of the arguments in Section \ref{sec:links3} (note that we do not need here the fact that $\psi'_1$ and $\psi_1$ are small).
There is a birational map $\sigma_{11} \colon X' \ratmap X$ which factorizes as
\[
\xymatrix{
Y' \ar[d]_{\varphi'} \ar@{-->}[r] & Y \ar[d]^{\varphi} \\
X' \ar@{-->}[r]^{\sigma_{11}} & X}
\]
where $\varphi'$ is the weighted blowup of $X'$ at $\msp'_1$ with $\wt (x_0,x_1,y_1,z) = \frac{1}{2} (1,1,4,3)$, $\varphi$ is the Kawamata blowup of $X$ at $\msp_1$ and $Y' \ratmap Y$ is a birational map.
The construction of the above birational map is possible in the case where the singularity of $X'$ at $\msp'_1$ is of type $cAx/2$ and that of $X$ at $\msp_1$ is of type $\frac{1}{4} (1,1,3)$.
Let $\Delta' \subset Y'$ and $\Delta \subset Y$ be proper transforms of $(y_1 = a_6 = c_8 = 0) \subset X'$ and $(y = s_1 = a_6 = c_8 = 0) \subset X$, respectively. 
Then the birational map $Y' \ratmap Y$ induces an isomorphism $Y' \setminus \Delta' \cong Y \setminus \Delta$.
We know that $X'$ has three singular points of type $\frac{1}{3} (1,1,2)$, $cAx/2$ and $cAx/2$.
Thus $Y'$ has three singular points whose types are $\frac{1}{3} (1,1,2)$, $\frac{1}{4} (1,1,3)$ and $cAx/2$ by the description of the weighted blowup.
By $Y' \setminus \Delta' \cong Y \setminus \Delta$, we see that $Y \setminus \Delta$ has at most three singular points of type $\frac{1}{3} (1,1,2)$, $\frac{1}{4} (1,1,3)$ and $cAx/2$.
It follows that $X \setminus (y = s_1 = a_6 = c_8 = 0)$ has also at most three singular points of type $\frac{1}{3} (1,1,2)$, $\frac{1}{4} (1,1,3)$ and $cAx/2$ since the center of $\varphi \colon Y \to X$ is contained in $(y = s_1 = a_6 = c_8 = 0)$. 
On the other hand, $X$ has singularities of type $\frac{1}{4} (1,1,3)$, $\frac{1}{4} (1,1,3)$ and $cAx/2$ at $\msp_1, \msp_2$ and $\msp_3$, respectively, and possibly Gorenstein singularities.
Therefore $X \setminus (y = s_1 = a_6 = c_8 = 0)$ has only singularities of type $\frac{1}{4} (1,1,4)$ and $cAx/2$ (at $\msp_2$ and $\msp_3$).
By changing the role of $s_0$ and $s_1$, we also see that $X \setminus (y = s_0 = a_6 = c_8 = 0)$ has only singularities of type $\frac{1}{4} (1,1,3)$ and $cAx/2$ (at $\msp_1$ and $\msp_3$).

It is then enough to show that $X$ is nonsingular along $S := (y = s_0 = s_1 = a_6 = c_8 = 0)$.
We see that the restriction to $S$ of the Jacobian matrix of the affine cone $C_X$ of $X$ can be written as
\[
J_{C_X}|_S =
\begin{pmatrix}
\frac{\prt a_6}{\prt x_0} & \frac{\prt a_6}{\prt x_1} & 0 & \frac{\prt a_6}{\prt z} & 0 & 0 \\[2mm]
-\frac{\prt c_8}{\prt x_0} & -\frac{\prt c_8}{\prt x_1} & 0 & -\frac{\prt c_8}{\prt z} & 0 & 0
\end{pmatrix}.
\]
Therefore $X$ is quasismooth along $S$ since the complete intersection $(a_6 = c_8 = 0)$ in $\mbP (1,1,3)$ is quasismooth for general $a_6$ and $c_8$ by Theorem \ref{qsmcriwci}.
Thus $X$ is nonsingular along $S$ and this completes the proof.
\end{proof}

\subsection{Structure of proof}

The remainder of this paper is devoted to a proof of the following.

\begin{Thm} \label{mainthm2}
Let $(a_6,b_6,c_8)$ be a triplet of homogeneous polynomials in $x_0,x_1,z$ satisfying \emph{Condition \ref{cond}} and $X'$, $X_1$, $X_2$ the $\mbQ$-Fano threefolds corresponding to $(a_6,b_6,c_8)$.
Then no nonsingular point and no curve on $X'$, $X_1$ and $X_2$ is a maximal center.
As for the Sarkisov links from $X'$, $X_1$ and $X_2$ centered at singular points, the following hold.
\begin{enumerate}
\item There exit Sarkisov links $X' \ratmap X_1$ and $X' \ratmap X_2$ centered at the $cAx/2$ points $\msp'_1$ and $\msp'_2$, respectively.
\item There exists a Sarkisov link $X' \ratmap X'$ centered at the $\frac{1}{3} (1,1,2)$ point $\msp'_3$ of $X'$ which is a birational involution.
\item For $i = 1,2$, there exists a Sarkisov link $X_i \ratmap X'$ centered at each $\frac{1}{4} (1,1,3)$ point of $X_i$.
\item For the $cAx/2$ points $\msp_3 \in X_1$ and $\msp_3 \in X_2$, one of the following holds.
\begin{enumerate}
\item Neither $\msp_3 \in X_1$ nor $\msp_3 \in X_2$ is a maximal center.
\item There exits a Sarkisov link $X_1 \ratmap X_2$ centered at $\msp_3 \in X_1$ and its inverse $X_2 \ratmap X_1$ is centered at $\msp_3 \in X_2$.
\end{enumerate}
\end{enumerate}
\end{Thm}

In view of the fact that there is a unique divisorial extraction centered at each singular point of $X'$, $X_1$ and $X_2$, Theorem \ref{mainthm} follows from Proposition \ref{prop:sym} and Theorem \ref{mainthm2} by \cite[Lemma 2.32]{Okada2}.
The construction of Sarkisov links will be given in Section \ref{sec:links} and exclusion of nonsingular points and curves as maximal centers will be done in Sections \ref{sec:exclX'} and \ref{sec:exclX_i}.
In Section \ref{sec:biraut}, we state the classification of Sarkisov links and give a description of the birational automorphism group.

\section{Sarkisov links} \label{sec:links}

We construct various Sarkisov links between $X'$, $X_1$ and $X_2$.
Throughout this section we assume that $(a_6,b_6,c_8)$ satisfies Condition \ref{cond}.

\subsection{Birational involution of $X'$} \label{sec:birinvX'}

We construct a birational involution $\iota'$ of $X'$ which is a Sarkisov link centered at the $\frac{1}{3} (1,1,2)$ point $\msp'_3$.
The construction is the same as that of \cite[Section 4.4]{CPR} to which we refer the readers for a detail.
After re-scaling $y_0,y_1,z$, we may assume that the coefficients of $z^2$ in $a_6$ and $b_6$ are both $1$.
We write $a_6 = z^2 + z f_3 + f_6$, $b_6 = z^2 + z g_3 + g_6$ and $c_8 = z^2 h_2 + z h_5 + h_8$, where $f_i, g_i, h_i \in \mbC [x_0,x_1]$.
It follows that the defining polynomial of $X'$ is 
\[
F' := (y_0+y_1+h_2) z^2 + (y_0 f_3 + y_1 g_3 + h_5)z + y_0^2 y_1^2 + y_0 f_6 + y_1 g_6 + h_8.
\]
Let $Z'$ be the weighted hypersurface in $\mbP (1,1,2,2,5)$ with homogeneous coordinates $x_0,x_1,y_0,y_1,t$, where $\deg t = 5$, defined by the equation
\[
t^2 + (y_0 f_3 + y_1 g_3 + h_5) t + (y_0 + y_1 + h_2)(y_0^2 y_1^2 + y_0 f_6 + y_1 g_6 + h_8) = 0.
\]
This equation is obtained by multiplying $F'$ by $y_0 + y_1 + h_2$ and then identifying $t$ with $(y_0+y_1+h_2)z$.
This identification gives rise to a birational map $X' \ratmap Z'$.
Let $\varphi' \colon Y' \to X'$ be the Kawamata blowup of $X'$ at $\msp'_3$.
Then $\varphi'$ resolves the indeterminacy of $X' \ratmap Z'$ and the induced birational morphism $\psi' \colon Y' \to Z'$ is a flopping contraction contracting the proper transform of the closed subscheme
\[
(y_0 + y_1 + h_2 = y_0 f_3 + y_1 g_3 + h_5 = y_0^2 y_1^2 + y_0 f_6 + y_1 g_6 + h_8 = 0)
\]
in $\mbP (1,1,2,2,3)$, which consists of finitely many curves by the argument of \cite[Section 4.4]{CPR} using quasismoothness.
Let $\iota_{Z'} \colon Z' \to Z'$ be the biregular involution interchanging the fibers of the double cover $Z' \to \mbP (1,1,2,2)$.
Then $\iota_{Y'} := {\psi'}^{-1} \circ \iota_{Z'} \circ \psi' \colon Y' \ratmap Y'$ is the flop and we have a Sarkisov link $\iota' = {\varphi'}^{-1} \circ \iota_{Y'} \circ \varphi' \colon X' \ratmap X'$.
In summary, we have

\begin{Prop}
The diagram
\[
\xymatrix{
Y' \ar[d]_{\varphi'} \ar@{-->}[r]^{\iota_{Y'}} & Y' \ar[d]^{\varphi'} \\
X' \ar@{-->}[r]_{\iota'} & X'}
\]
is a Sarkisov link centered at $\msp'_3$ that is a birational involution.
\end{Prop}

\subsection{Link between $X_1$ and $X_2$}

For $i = 1,2$, let $\varphi_i \colon Y_i \to X_i$ be the weighted blowup of $X_i$ at the $cAx/2$ point $\msp_3$ with $\wt (x_0,x_1,z,s_0,s_1) = \frac{1}{2} (1,1,3,4,4)$ and $\pi_i \colon X_i \ratmap \mbP (1,1,3,4,4)$ the projection with coordinates $x_0,x_1,z,s_0$ and $s_1$.
The images of $\pi_1$ and $\pi_2$ are the same and it is the weighted hypersurface 
\[
Z := ((s_0 + s_1)(s_0 s_1 - c_8) + a_6 b_6 = 0) \subset \mbP (1,1,3,4,4).
\]
The sections $x_0,x_1,z,s_0$ and $s_1$ on $X_i$ lift to plurianticanonical sections on $Y_i$ and they define the morphism $\psi_i \colon Y_i \to Z$ such that $\psi_i = \varphi_i \circ \pi_i$.
It follows that $\psi_i$ is a $K_{Y_i}$-trivial contraction.
We see that $\psi_i$ contracts the proper transform on $Y_i$ of
\[
\Delta := (s_0 + s_1 = s_0 s_1 - c_8 = a_6 = b_6 = 0) \subset \mbP (1,1,2,3,4,4).
\]
We see that $\dim \Delta = 2$ if and only if $a_6 \sim b_6$ since the projection $\Delta \to \mbP (1,1,2,3)$ is a finite morphism (of degree $2$) onto $(a_6 = b_6 = 0) \subset \mbP (1,1,2,3)$. 
Here, $a_6 \sim b_6$ means that $a_6$ is proportional to $b_6$, that is, there is a non-zero $\lambda \in \mbC$ such that $a_6 = \lambda b_6$.

\begin{Lem}
If $a_6 \sim b_6$, then the $cAx/2$ point of $X_i$ is not a maximal center for $i = 1,2$.
\end{Lem}

\begin{proof}
Since $\varphi_i$ is a unique divisorial extraction centered at the $cAx/2$ point of $X_i$ by Lemma \ref{lem:uniqueextr} (2), it is enough to show that $\varphi_i$ is not a maximal extraction.
We have $K_{Y_i} = \varphi_i^*K_{X_i} + (1/2)E_i$, where $E_i$ is the exceptional divisor of $\varphi_i$.
Note that $\Delta$ is a surface since $a_6 \sim b_6$.
It follows that $\psi_i$ contracts a divisor.
Let $C$ be an irreducible and reduced curve on $Y_i$ contracted by $\psi_i$.
Then, $(-K_{Y_i} \cdot C) = 0$ and 
\[
(E_i \cdot C) = 2 (K_{Y_i} \cdot C) - 2 (\varphi_i^*K_{X_i} \cdot C)
= - 2 (\varphi_i^*K_{X_i} \cdot C) > 0
\]
since $C$ is not contracted by $\varphi_i$.
This shows that there are infinitely many curves on $Y_i$ which intersect $-K_{Y_i}$ non-positively and $E_i$ positively.
It follows from \cite[Lemma 2.20]{Okada2} that $\varphi_i$ is not a maximal extraction.
\end{proof}

\begin{Prop} \label{prop:theta}
Assume that $a_6 \not\sim b_6$.
Then the diagram
\[
\xymatrix{
Y_1 \ar[d]_{\varphi_1} \ar@{-->}[rr]^{\psi_2^{-1} \circ \psi_1} \ar[rd]^{\psi_1} & & Y_2 \ar[ld]_{\psi_2} \ar[d]^{\varphi_2} \\
X_1 & Z & X_2}
\] 
gives a Sarkisov link $\theta \colon X_1 \ratmap X_2$ centered at the $cAx/2$ point of $X_1$.
The inverse $\theta^{-1} \colon X_2 \ratmap X_1$ is a Sarkisov link centered at the $cAx/2$ point of $X_2$.
\end{Prop}

\begin{proof}
By the assumption, $\dim \Delta = 1$ and thus $\psi_i$ is a flopping contraction since $\psi_i$ is a $K_{Y_i}$-trivial contraction whose exceptional locus is the proper transform of $\Delta \subset X_i$.
The birational map $\theta = \pi_2^{-1} \circ \pi_1 \colon X_1 \ratmap X_2$ is given by
\[
(x_0 \!:\! x_1 \!:\! y \!:\! z \!:\! s_0 \!:\! s_1) \mapsto (x_0 \!:\! x_1 \!:\! \frac{b_6}{a_6} y \!:\! z \!:\! s_0 \!:\! s_1).
\]
We claim that $\theta$ is not biregular.
For $i = 1,2$, let $E_i$ be the exceptional divisor $\varphi_i$.
We see that $E_1$ is isomorphic to the weighted complete intersection
\[
(s_0 + s_1 = b_6 = 0) \subset \mbP (1,1,3,4,4)
\]
and $\psi_1|_{E_1} \colon E_1 \to Z$ can be identified with the restriction of the identity mapping of $\mbP (1,1,3,4,4)$.
It follows that $\psi_1 (E_1) = (s_0 + s_1 = b_6 = 0)$ and, similarly, $\psi_2 (E_2) = (s_0 + s_1 = a_6 = 0)$.
Since $a_6 \not\sim b_6$, we have $\psi_1 (E_1) \ne \psi_2 (E_2)$.
This implies that $\theta$ (resp.\ $\theta^{-1}$) contracts the birational transform on $X_1$ (resp.\ $X_2$) of $\psi_2 (E_2)$ (resp.\ $\psi_1 (E_1)$).
Thus $\theta$ is not biregular.
It follows that $\psi_2^{-1} \circ \psi_1 \colon Y_1 \ratmap Y_2$ is a flop and thus $\theta \colon X_1 \ratmap X_2$ is a Sarkisov link.
\end{proof}

\begin{Rem} \label{rem:thetadiv}
We make explicit the description of the proper transform of $E_2$ on $X_1$ for the later use.
By the proof of Proposition \ref{prop:theta}, it is the divisor on $X_1$ which maps onto $\psi_2 (E_2) = (s_0+s_1 = a_6 = 0)$ via the projection $\pi_2 \colon X_2 \ratmap Z$, which must be the divisor $(s_0+s_1 = 0)_{X_1} \subset X_1$.
\end{Rem}

\begin{Rem}
Note that $a_6 \sim b_6$ implies that $(a_6,b_6,c_8)$ is symmetric (we do not know whether or not the converse holds).
It follows that $X_1$ and $X_2$ are connected by a Sarkisov link whenever $X_1$ is not isomorphic to $X_2$.
\end{Rem}

Note that if $(a_6,b_6,c_8)$ is asymmetric, then $\theta \colon X_1 \ratmap X_2$ is a Sarkisov link between non-biregularly equivalent $\mbQ$-Fano threefolds, but if $(a_6,b_6,c_8)$ is symmetric and $a_6 \not\sim b_6$, then $\theta$ is a birational involution of $X = X_1 \cong X_2$.

\subsection{Links between $X'$ and $X_i$} \label{sec:links3}

We construct Sarkisov links between $X'$ and $X_i$ for $i = 1,2$.
Recall that
\[
\msp'_1 = (0 \!:\! 0 \!:\! 1 \!:\! 0 \!:\! 0) \text{ and } \msp'_2 = (0 \!:\! 0 \!:\! 0 \!:\! 1 \!:\! 0)
\]
are the $cAx/2$ points of $X'$ and
\[
\msp_1 = (0 \!:\! 0 \!:\! 0 \!:\! 0 \!:\! 1 \!:\! 0) \text{ and } \msp_2 = (0 \!:\! 0 \!:\! 0 \!:\! 0 \!:\! 0 \!:\! 1)
\]
are the $\frac{1}{4} (1,1,3)$ points of $X_i$.
Let $\mbP := \mbP (1,1,2,3,4)$ be the weighted projective space with homogeneous coordinates $x_0,x_1,y,z,s$ and let $\pi'_1 \colon X' \ratmap \mbP$ be the rational map defined by 
\[
(x_0 \!:\! x_1 \!:\! y_0 \!:\! y_1 \!:\! z) \mapsto (x_0 \!:\! x_1 \!:\! y_1 \!:\! z \!:\! y_0 y_1).
\]
By multiplying the defining polynomial of $X'$ by $y_1$ and then replacing $y_1$ with $y$ and $y_0 y_1$ with $s$, we see that the image of $\pi'_1$ is the weighted hypersurface
\[
Z_1 = (s^2 y + s a_6 + y^2 b_6 + y c_8 = 0) \subset \mbP (1,1,2,3,4),
\]
and $\pi'_1 \colon X' \ratmap Z_1$ is a birational map defined outside $\msp'_1$.

Let $\pi_1 \colon X_1 \ratmap \mbP$ be the projection defined by
\[
(x_0 \!:\! x_1 \!:\! y \!:\! z \!:\! s_0 \!:\! s_1) \mapsto (x_0 \!:\! x_1 \!:\! y \!:\! z \!:\! s_1),
\]
which is defined outside $\msp_1$.
By considering the ratio
\[
s_0 = - \frac{s_1 y + a_6}{y} = \frac{y b_6 + c_8}{s_1},
\]
we see that the image of $\pi_1$ is $Z_1$ and $\pi_1 \colon X_1 \ratmap Z_1$ is birational.
We define $\sigma_{11} := \pi_1^{-1} \circ \pi'_1 \colon X' \ratmap X_1$. 

Let $\eta_1 \colon X_1 \to X_1$ be the automorphism of $X_1$ which interchanges $s_0$ and $s_1$ and we define $\sigma_{12} := \eta_1 \circ \sigma_{11} \colon X' \ratmap X_1$.
By the symmetry between $y_0$ and $y_1$, the same construction gives a birational map $\sigma_{21} \colon X' \ratmap X_2$ and $\sigma_{22} := \eta_2 \circ \sigma_{21} \colon X' \ratmap X_2$, where $\eta_2$ is the automorphism of $X_2$ which interchanges $s_0$ and $s_1$.

\begin{Prop}
For $i = 1,2$ and $j = 1,2$, the birational map $\sigma_{i j} \colon X' \ratmap X_i$ is a Sarkisov link centered at the $cAx/2$ point $\msp'_i$ and the inverse $\sigma_{ij}^{-1} \colon X_i \ratmap X'$ is a Sarkisov link centered at the $\frac{1}{4} (1,1,3)$ point $\msp_j$.
\end{Prop}

\begin{proof}
We prove the assertion for $\sigma_{11}$.
The rest follows by symmetry.

Let $\varphi'_1 \colon Y'_1 \to X'$ be the weighted blowup of $X'$ at $\msp'_1$ with $\wt (x_0,x_1,y_1,z) = \frac{1}{2} (1,1,4,3)$.
Note that $\varphi'_1$ is a unique divisorial extraction of centered at $\msp'_1$.
We see that $x_0$, $x_1$, $y_1$, $z$ and $y_0 y_1$ lift to plurianticanonical sections on $Y'$ and $\varphi'_1$ resolves the indeterminacy of $\pi'_1$.
Thus we have a $K_{Y'}$-trivial birational morphism $\psi'_1 \colon Y' \to Z$.
Let $\varphi_1 \colon Y_1 \to X_1$ the Kawamata blowup of $X_1$ at $\msp_1$.
We see that $x_0,x_1,y,z,s_1$ lift to plurianticanonical sections on $Y_1$ and $\varphi_1$ resolves the indeterminacy of $\pi_1$.
Thus we have a $K_{Y_1}$-trivial birational morphism $\psi_1 \colon Y_1 \to Z$ and the diagram
\[
\xymatrix{
Y'_1 \ar[d]_{\varphi'_1} \ar@{-->}[rr] \ar[rd]^{\psi'_1} & & Y_1 \ar[d]^{\varphi_1} \ar[ld]_{\psi_1} \\
X' & Z & X_1}
\]
We will show that $\psi'_1$ and $\psi_1$ are small contractions.
Then $Y'_1 \ratmap Y_1$ is the flop since $\rho (Y'_1) = \rho (Y_1) = 2$ and $Y'_1$ and $Y_1$ are not isomorphic over $Z$ (if $Y'_1$ and $Y_1$ are isomorphic over $Z$, then $X'_1 \cong X_1$.
This is absurd since they have different singularities).

We see that $\psi'_1$ contracts the proper transform of $(y_1 = a_6 = c_8= 0) \subset X'$ to $S := (y = a_6 = c_8 = s = 0) \subset Z$, and $\psi_1$ contracts the proper transform of $(y = s_1 = a_6 = c_8 = 0) \subset X_1$ to $S$.
Therefore $\psi'_1$ is divisorial if and only if $\psi_1$ is so, and this is equivalent to the assertion that $a_6$ and $c_8$ share a common component.
Assume that $a_6$ and $c_8$ have a component $d \in \mbC [x_0,x_1,z]$.
Then, since $(a_6 = 0) \subset \mbP (1,1,3)$ is quasismooth, the polynomial $a_6$ is irreducible and we may assume $d = a_6$.
Hence $c_8 = a_6 e_2$ for some $e_2 \in \mbC [x_0,x_1,z]$.
Let $C = (y = s_0 = s_1 = a_6 = 0)$ be a curve.
We see that $C \subset X_1$ and the restriction of the Jacobian matrix of the affine cone of $X_1$ to $C$ is of the form
\[
J_{C_{X_1}} |_C =
\begin{pmatrix}
\frac{\prt a_6}{\prt x_0} & \frac{\prt a_6}{\prt x_1} & 0 & \frac{\prt a_6}{\prt z} & 0 & 0 \\[2mm]
-\frac{\prt a_6}{\prt x_0} e_2 & -\frac{\prt a_6}{\prt x_1} e_2 & -b_6 & - \frac{\prt a_6}{\prt z} e_2 & 0 & 0 \\
\end{pmatrix}.
\]
This shows that $X_1$ is not quasismooth along $C \cap (b_6 = 0)$.
This is a contradiction and thus $Y'_1 \ratmap Y_1$ is a flop.
\end{proof}

\begin{Rem}
In the above proof, the fact that $\psi'_1$ and $\psi$ are small contractions follows from the following more conceptual argument.
Both $Y'_1$ and $Y_1$ are crepant $\mbQ$-factorial terminalizations of $Z$.
Hence, by a general fact, they are either isomorphic or connected by a sequence of flops.
But they cannot be isomorphic as is explained in the above proof. 
It follows that $Y'_1$ and $Y_1$ admit at least one flopping contraction.
But since they have Picard number $2$, $\psi'_1$ and $\psi_1$ must be flopping contractions.
\end{Rem} 

\section{Excluding maximal centers on $X'$} \label{sec:exclX'}

In this section let $(a_6,b_6,c_8)$ be a triplet satisfying Condition \ref{cond}.
We exclude all the nonsingular points and curves on $X'$ as maximal singularity.

\subsection{Nonsingular points}

\begin{Def}
Let $X$ be a normal projective variety embedded in a weighted projective space $\mbP (a_0,\dots,a_n)$ with homogeneous coordinates  $x_0,\dots,x_n$ and $\msp \in X$ a nonsingular point.
We say that a set $\{g_i\}$ of homogeneous polynomials in $x_0,\dots,x_n$ {\it isolates} $\msp$ if $\msp$ is an isolated component of
\[
X \cap \bcap_i (g_i = 0).
\] 

We say that a Weil divisor $L$ {\it isolates} $\msp$ if there is an integer $s > 0$ such that $\msp$ is an isolated component of the base locus of the linear system
\[
\mcL^s_{\msp} := \left| \mcI_{\msp}^s (s L) \right|.
\]
\end{Def}

\begin{Lem}[{\cite{CPR}}] \label{lem:exclcrinspt}
Let $X$ be a $\mbQ$-Fano $3$-fold with Picard number one and $\msp \in X$ a nonsingular point.
If $-l K_X$ isolates $\msp$ for some $l \le 4 / (-K_X)^3$, then $\msp$ is not a maximal center.
\end{Lem}

\begin{proof}
We refer the reader to \cite[Proof of (A)]{CPR} and also to \cite[Lemma 2.14]{Okada2} for a proof.
\end{proof}

The following enables us to find a divisor which isolates a nonsingular point.

\begin{Lem}[{\cite[Lemma 5.6.4]{CPR}}]
Let $X$ be a normal projective variety embedded in $\mbP (a_0,\dots,a_n)$ and $\{g_i\}$ a set of homogeneous polynomials of $\deg g_i = l_i$.
If a set $\{g_i\}$ of polynomials isolates $\msp$, then $l A$ isolates $\msp$, where $l = \max \{l_i\}$ and $A$ is a Weil divisor on $X$ such that $\mcO_X (A) \cong \mcO_X (1)$.
\end{Lem}

\begin{Prop} \label{prop:exclnsptX'}
No nonsingular point on $X'$ is a maximal center.
\end{Prop}

\begin{proof}
Let $\msp = (\xi_0 \!:\! \xi_1 \!:\! \eta_0 \!:\! \eta_1 \!:\! \zeta)$ be a nonsingular point of $X'$.
If $\xi_0 \ne 0$, then the set 
\[
\{ \xi_0 x_1 - \xi_0 x_0, \xi_0^2 y_0 -  \eta_0 x_0^2, \xi_0^2 y_1 - \eta_1 x_0^2, \xi_0^3 z - \zeta x_0^3 \}
\]
isolates $\msp$ and thus $- 3 K_{X'}$ isolates $\msp$.
Similarly, $- 3 K_{X'}$ isolates $\msp$ if $\xi_1 \ne 0$.
Assume that $\xi_0 = \xi_1 = 0$.
In this case, at least one of $\eta_0$ and $\eta_1$ is non-zero since $\msp$ is not a singular point.
Without loss of generality, we may assume $\eta_0 \ne 0$.
Then the set
\[
\{ x_0,x_1,\eta_0 y_1 - \eta_1 y_0, \eta_0^3 z^2 - \zeta^2 y_0^3 \}
\]
isolates $\msp$ and thus $- 6 K_{X'}$ isolates $\msp$.
Therefore Lemma \ref{lem:exclcrinspt} shows that $\msp$ is not a maximal center since $3 < 6 \le 4/(-K_{X'})^3 = 6$.
\end{proof}

\subsection{Curves}

The aim of this subsection is to show that no curve on $X'$ is a maximal center.
The following excludes most of the curves on $X'$ as maximal centers.

\begin{Lem} \label{lem:exclcurveX'}
No curve on $X'$ is a maximal center except possibly for a curve of degree $1/2$ which does not pass through the $\frac{1}{3} (1,1,2)$ point $\msp'_3$.
\end{Lem}

\begin{proof}
Let $\Gamma \subset X'$ be a curve.
By \cite[Lemma 2.9]{Okada2}, $\Gamma$ can be a maximal center only if $(-K_{X'} \cdot \Gamma) < (-K_{X'})^3 = 2/3$.
If $\Gamma$ passes through the $\frac{1}{3} (1,1,2)$ point $\msp'_3$, then it is not a maximal singularity since there is no divisorial extraction centered along a curve passing through a terminal quotient singular point (\cite{Kaw96}).
If $\Gamma$ does not pass through $\msp'_3$, then $(-K_{X'} \cdot \Gamma) \in \frac{1}{2} \mbZ$.
This follows since the divisor $(y_0 + y_1 = 0)_{X'} \sim_{\mbQ} - 2 K_{X'}$ intersects $\Gamma$ at nonsingular points of $X'$ and thus $(- 2 K_{X'} \cdot \Gamma) \in \mbZ$.
Combining the above arguments, $\Gamma$ is not a maximal center unless it satisfies $(- K_{X'} \cdot \Gamma) = 1/2$ and $\msp'_3 \not\in\Gamma$.
\end{proof}

Let $\Gamma$ be a curve of degree $1/2$ on $X'$ which does not pass through $\msp'_3$.
Since $\Gamma$ passes through a $cAx/2$ point, we may assume $\msp'_1 \in \Gamma$ without loss of generality.
The defining polynomial of $X'$ is $F' := y_0^2 y_1^2 + y_0 a_6 + y_1 b_6 + c_8$.
After re-scaling $y_0, y_1, z$, we may assume that the coefficients of $z^2$ in $a_6$ and $b_6$ are both $1$. 

\begin{Lem} \label{lem:defeqGamma}
We have $\Gamma = (x_1 = y_1 = z = 0)$ after replacing $x_0,x_1,z$.
\end{Lem}

\begin{proof}
The restriction $\pi|_{\Gamma} \colon \Gamma \to \pi(\Gamma)$ of the projection $\pi \colon X' \ratmap \mbP (1,1,2,2)$ from $\msp'_3$ is a finite morphism since $\msp'_3 \notin \Gamma$.
We have $1/2 = \deg (\pi|_{\Gamma}) \deg (\pi (\Gamma))$ and $\deg \pi (\Gamma) \in \frac{1}{4} \mbZ$. 
We claim that $\deg \pi (\Gamma) = 1/2$.
Indeed, if $\deg \pi (\Gamma) = 1/4$, then $\pi (\Gamma) = (x_0 = x_1 = 0)$.
It follows that 
\[
\Gamma \subset (x_0 = x_1 = 0)_{X'} = (x_0 = x_1 = y_0^2 y_1^2 + y_0 z^2 + y_1 z^2 = 0).
\]
We see that $(x_0 = x_1 = 0)_{X'}$ is an irreducible and reduced curve of degree $2/3$.
This is a contradiction and the claim is proved.

After replacing $x_0,x_1$, we may assume that $\pi (\Gamma) = (x_1 = \theta_0 y_0 + \theta_1 y_1 - \lambda x_0^2 = 0)$ for some $\theta_0,\theta_1,\lambda \in \mbC$.
Since $\msp'_1 \in \Gamma$, $\pi (\Gamma)$ passes through $(0 \!:\! 0 \!:\! 1 \!:\! 0) \in \mbP (1,1,2,2)$.
This implies that $\theta_0 = 0$ and then we may assume that $\theta_1 = 1$.
Since $\deg \Gamma = 1/2$ and $\Gamma \subset (x_0 = y_1 - \lambda x_0^2 = 0)_{X'}$, we have $\Gamma = (x_1 = y_1 - \lambda x_0^2 = z - \mu y_0 x_0 - \nu x_0^3 = 0)$ for some $\mu, \nu \in \mbC$.
Replacing $z \mapsto z + \nu x_0^3$, we assume $\nu = 0$.
Now it is straightforward to see that $\Gamma$ is indeed contained in $X'$ if and only if $\lambda = \mu = 0$, $x_0^6 \notin a_6$ and $x_0^8 \notin c_8$. 
This completes the proof.
\end{proof}

We write $a_6 = z^2 + z f_3 (x_0,x_1) + f_6 (x_0,x_1)$.
Then, by the proof of Lemma \ref{lem:defeqGamma}, we have $f_6 (x_0,0) = c_8 (x_0,0,0) = 0$ since $\Gamma = (x_1 = y_1 = z = 0)$ is contained in $X'$.
We write $f_6 = x_1 f_5$.

\begin{Lem}
At least one of $f_3$ and $f_5$ is not divisible by $x_1$.
\end{Lem}

\begin{proof}
Let $F_1 := s_0 y + s_1 y + a_6$ be the defining polynomial of $X_1$ of degree $6$.
If both $f_3$ and $f_5$ are divisible by $x_1$, then $\prt F_1/\prt x_0$, $\prt F_1/\prt x_1$, $\prt F_1/\prt y$, $\prt F_1/\prt z$, $\prt F_1/\prt s_0$ and $\prt F_1/\prt s_1$ vanish at the point $(1 \!:\! 0 \!:\! 0 \!:\! 0 \!:\! 0 \!:\! 0) \in X_1$.
This is a contradiction since $X_1$ is quasismooth outside its $cAx/2$ point. 
\end{proof}

Let $\mcM \subset \left| - 3 K_{X'} \right|$ be the linear system spanned by the cubic monomials vanishing along $\Gamma$ other than $y_0 x_1$, namely, the sections $x_0^2 x_1$, $x_0 x_1^2$, $x_1^3$, $y_1 x_0$, $y_1 x_1$, $z$, and let $S$ be a general member of $\mcM$.
We have $\Bs \mcM = \Gamma \cup \{\msp'_2\}$, $\Bs \mcM_{y_1} = (x_0 = x_1 = 0)_{X'} \not\supset \Gamma$, $\Bs \mcM_{x_1} = (x_0 = x_1 = y_1 = 0)_{X'}$.
Thus, by Proposition \ref{prop:qsmcurve}, $S$ is nonsingular along $\Gamma \setminus \{\msp'_1\}$.

\begin{Lem} \label{gammaselfint}
We have $(\Gamma^2) \le -3/2$.
\end{Lem}

\begin{proof}
The section which cuts out $S$ on $X'$ can be written as $z + x_1 q + \alpha_0 y_1 x_0 + \alpha_1 y_1 x_1$, where $q = q (x_0,x_1)$ is a quadric and $\alpha_0,\alpha_1 \in \mbC$.
We work on the open subset on which $y_0 \ne 0$.
Let $\varphi \colon T \to S$ be the weighted blowup of $S$ at $\msp'_1$ with $\wt (x_0,x_1,y_1,z) = \frac{1}{2} (1,1,4,3)$, $E$ its exceptional divisor and $\tilde{\Gamma}$ the proper transform of $\Gamma$ on $T$.
We claim that $E = E_1 + E'$, where $E_1$ is a prime divisor, $E'$ does not contain $E_1$ as a component, $(\tilde{\Gamma} \cdot E_1) = 1$ and $\tilde{\Gamma}$ is disjoint from the support of $E'$.
Indeed we have the isomorphisms
\[
\begin{split}
E &\cong (z^2 + z f_3 + x_1 f_5 = z + x_1 q = 0) \subset \mbP (1,1,4,3) \\
&\cong (x_1^2 q^2 - x_1 q f_3 + x_1 f_5 = 0) \subset \mbP (1,1,4).
\end{split}
\] 
We set $E_1 = (x_1 = 0)$ and $E' = (x_1 q^2 - q f_3 + f_5 = 0)$.
Since at least one of $f_3$ and $f_5$ is not divisible by $x_1$ and $q$ is general, we see that $E'$ does not contain $E_1$ as a component and $E'$ is disjoint from $\tilde{\Gamma}$.
Moreover, $E_1$ intersects $\tilde{\Gamma}$ transversally at a nonsingular point.
This proves the claim.

We write $\varphi^*\Gamma = \tilde{\Gamma} + r E_1 + F$ for some rational number $r$ and an effective $\mbQ$-divisor $F$ whose support is contained in $\operatorname{Supp} E'$.
We have $r \le 1/2$ since the section $x_1$ cuts out on $S$ the union of  the curve $\Gamma$ and another curve, and $x_1$ vanishes along $E_1$ to order $1/2$.
An explicit computation shows that $K_T = \varphi^*K_S - E$ and we see that $\tilde{\Gamma} \cong \mbP^1$.
We have 
\[
(\Gamma^2) = (\varphi^*\Gamma \cdot \tilde{\Gamma}) = (\tilde{\Gamma}^2) + (r E_1+F \cdot \tilde{\Gamma}) = (\tilde{\Gamma}^2) + r
\] 
and 
\[
(\tilde{\Gamma}^2) = - (K_T\cdot \tilde{\Gamma}) - 2 = - (K_S \cdot \Gamma) - 1.
\]
Combining these with $(K_S \cdot \Gamma) = 2 \deg \Gamma = 1$, we get $(\Gamma^2) = -2 + r \le -3/2$.
\end{proof}

\begin{Prop} \label{prop:exclcurveX'}
No curve on $X'$ is a maximal center.
\end{Prop}

\begin{proof}
By Lemma \ref{lem:exclcurveX'}, it is enough to exclude a curve $\Gamma$ of degree $1/2$ which does not pass through $\msp'_3$.
We keep the above notation.
We assume that $\Gamma$ is a maximal center.
An extremal divisorial extraction (between terminal $3$-folds) centered along a curve is unique, if it exists, and it is generically the blowup along $\Gamma$.
Hence there is a movable linear system $\mcH \subset \left| - n K_{X'} \right|$ on $X'$ such that $\mult_{\Gamma} \mcH > n$.
Let $S$ be a general member of $\mcM$ so that we have
\[
(-K_{X'})|_S \sim_{\mbQ} \frac{1}{n} \mcH|_S = \frac{1}{n} \mcL + \gamma \Gamma,
\]
where $\mcL$ is the movable part of $\mcH|_S$ and $\gamma \ge \mult_{\Gamma} \mcH/n > 1$.
This is possible since the base locus of $\mcM$ does not contain a curve other than $\Gamma$.
Let $L$ be a $\mbQ$-divisor on $S$ such that $n L \in \mcL$.
Note that $(L^2) \ge 0$ since $L$ is nef.
We get
\[
(L^2) = (-K_{X'}|_S - \gamma \Gamma)^2 = 3 (-K_{X'})^3 - 2 (\deg \Gamma) \gamma + (\Gamma^2) \gamma^2 = 2 - \gamma + (\Gamma^2) \gamma^2.
\]
Since $(\Gamma^2) < -3/2$ by Lemma \ref{gammaselfint} and $\gamma > 1$, we have
\[
(L^2) < 2 - 1 + (\Gamma^2) \le - 1/2.
\]
This is a contradiction and $\Gamma$ is not a maximal center.
\end{proof}

\section{Excluding maximal centers on $X_1$ and $X_2$} \label{sec:exclX_i}

In this section let $(a_6,b_6,c_8)$ be a triplet satisfying Condition \ref{cond}.
We exclude nonsingular points and curves on $X$, where $X$ is either $X_1$ or $X_2$.

\subsection{Nonsingular points}

\begin{Prop} \label{prop:exclnsptX}
No nonsingular point on $X$ is a maximal center.
\end{Prop}

\begin{proof}
We show that $-4K_X$ isolates $\msp$.
Let $\msp = (\xi_0 \!:\! \xi_1 \!:\! \eta \!:\! \zeta \!:\! \sigma_0 \!:\! \sigma_1)$ be a nonsingular point of $X$.
If $\xi_0 \ne 0$, then the set
\[
\{ \xi_1 x_0 - \xi_0 x_1, \xi_0^2 y - \eta x_0^2, \xi_0^3 z - \zeta x_0^3, \xi_0^4 s_0 - \sigma_0 x_0^4, \xi_0^4 s_1 - \sigma_1 x_0^4 \}
\]
isolates $\msp$ and thus $- 4 K_X$ isolates $\msp$.
Similarly, if $\xi_1 \ne 0$, then $- 4 K_X$ isolates $\msp$.
Assume that $\xi_0 = \xi_1 = 0$.
If further $\eta = 0$, then $\msp$ is a singular point of type $\frac{1}{4} (1,1,3)$.
Hence $\eta \ne 0$ and the set
\[
\Lambda := \{ x_0,x_1, \eta^2 s_0 - \sigma_0 y^2, \eta^2 s_1 - \sigma_1 y^2 \}
\]
isolates $\msp$.
It follows that $- 4 K_X$ isolates $\msp$.
By Lemma \ref{lem:exclcrinspt}, $\msp$ is not a maximal center since $4 < 4/(-K_X)^3 = 8$.
\end{proof}

\subsection{Curves}

\begin{Prop} \label{prop:exclcurveX}
No curve on $X$ is a maximal center.
\end{Prop}

\begin{proof}
Let $\Gamma$ be an irreducible curve on $X$.
If $\Gamma$ passes through a singular point of type $\frac{1}{4} (1,1,3)$, then there is no divisorial extraction centered along $\Gamma$ (\cite{Kaw96}), hence $\Gamma$ cannot be a maximal center.
If $\Gamma$ does not pass through a $\frac{1}{4} (1,1,3)$ point, then $(- 2 K_X \cdot \Gamma)$ is a positive integer and thus $(-K_X \cdot \Gamma) \ge 1/2$.
By \cite[Lemma 2.9]{Okada2}, $\msp$ is not a maximal center since $(-K_X)^3 = 1/2$.
This completes the proof.
\end{proof}

\section{Sarkisov links and the birational automorphism group of $X'$} \label{sec:biraut}

Throughout this section, we assume that $(a_6,b_6,c_8)$ satisfies Condition \ref{cond}.
We state a classification result of Sarkisov links and give a description of the birational automorphism group of $X'$.

By the construction given in Section \ref{sec:links}, explicit descriptions of links $\sigma_{ij}$ and $\theta$ between $X'$, $X_1$ and $X_2$ are given as follows:
\[
\begin{split}
\sigma_{11} & \colon X' \ratmap X_1, \  (x_0 \!:\! x_1 \!:\! y_0 \!:\! y_1 \!:\! z) \mapsto (x_0 \!:\! x_1 \!:\! y_1 \!:\! z \!:\! -y_0 y_1 - \frac{a_6}{y_1} \!:\! y_0 y_1), \\
\sigma_{11}^{-1} & \colon X_1 \ratmap X', \ (x_0 \!:\! x_1 \!:\! y \!:\! z \!:\! s_0 \!:\! s_1) \mapsto
(x_0 \!:\! x_1 \!:\! \frac{s_1}{y} \!:\! y \!:\! z), \\
\sigma_{21} & \colon X' \ratmap X_2, \ (x_0 \!:\! x_1 \!:\! y_0 \!:\! y_1 \!:\! z) \mapsto (x_0 \!:\! x_1 \!:\! y_0 \!:\! z \!:\! -y_0 y_1 - \frac{b_6}{y_0} \!:\! y_0 y_1), \\
\sigma_{21}^{-1} & \colon X_2 \ratmap X', \ (x_0 \!:\! x_1 \!:\! y \!:\! z \!:\! s_0 \!:\! s_1) \mapsto
(x_0 \!:\! x_1 \!:\! y \!:\!  \frac{s_1}{y} \!:\! z), \\
\theta & \colon X_1 \ratmap X_2, \ (x_0 \!:\! x_1 \!:\! y \!:\! z \!:\! s_0 \!:\! s_1) \mapsto (x_0 \!:\! x_1 \!:\! \frac{b_6}{a_6} y \!:\! z \!:\! s_0 \!:\! s_1), \\
\iota' & \colon X' \ratmap X', \ (x_0 \!:\! x_1 \!:\! y_0 \!:\! y_1 \!:\! z) \mapsto (x_0 \!:\! x_1 \!:\! y_0 \!:\! y_1 \!:\! - z - \frac{y_0 f_3 + y_1 g_3 + h_5}{y_0+y_1+h_2}),
\end{split}
\]
where $f_3,g_3, h_2, h_5$ are the polynomials defined in Section \ref{sec:birinvX'}.
See Section \ref{sec:links3} (resp.\ the proof of Proposition \ref{prop:theta}, resp.\ Section \ref{sec:birinvX'}) for the descriptions of $\sigma_{11}, \dots,\sigma_{21}^{-1}$ (resp.\ $\theta$, resp.\ $\iota'$).
We also defined $\sigma_{i2} = \eta_i \circ \sigma_{i1}$ and $\sigma_{i2}^{-1} = \sigma_{i1}^{-1} \circ \eta_i$ for $i = 1,2$, where $\eta_i$ is the biregular involution of $X_i$ interchanging $s_0$ and $s_1$.

\begin{Def}
In the case where $(a_6,b_6,c_8)$ is symmetric, we set $X := X_1 \cong X_2$ and $\sigma_j := \sigma_{1j}$ for $j = 1,2$. 
We define the set of Sarkisov links as
\[
\Sigma :=
\begin{cases}
\{ \sigma_{11}^{\pm}, \sigma_{12}^{\pm}, \sigma_{21}^{\pm}, \sigma_{22}^{\pm}, \theta^{\pm}, \iota' \}, & \text{if $(a_6,b_6,c_8)$ is asymmetric}, \\
\{ \sigma_1^{\pm}, \sigma_2^{\pm}, \theta^{\pm}, \iota' \}, & \text{if $(a_6,b_6,c_8)$ is symmetric and $a_6 \not\sim b_6$}, \\
\{ \sigma_1^{\pm}, \sigma_2^{\pm}, \iota' \}, & \text{if $a_6 \sim b_6$}.
\end{cases}
\]
\end{Def}

\begin{Thm}
The links in $\Sigma$ are all the Sarkisov links between the birational Mori fiber structures of $X'$.
\end{Thm}

\begin{proof}
This follows from Theorem \ref{mainthm2}.
\end{proof}

We have the following relations
\[
\begin{split}
\eta'_i &:= \sigma^{-1}_{i2} \circ \sigma_{i1} = \sigma^{-1}_{i1} \circ \sigma_{i2}, \\
\theta' &:= \sigma^{-1}_{21} \circ \theta \circ \sigma_{11} = \sigma^{-1}_{22} \circ \theta \circ \sigma_{12} = \sigma^{-1}_{11} \circ \theta^{-1} \circ \sigma_{21} = \sigma^{-1}_{12} \circ \theta^{-1} \circ \sigma_{22},
\end{split}
\]
where $\eta'_i$, $i = 1,2$, and $\theta'$ are birational involutions of $X'$ whose explicit descriptions are given as follows:
\[
\begin{split}
\eta'_1 \colon &(x_0 \!:\! x_1 \!:\! y_0 \!:\! y_1 \!:\! z) \mapsto (x_0 \!:\! x_1 \!:\! -y_0 - \frac{a_6}{y_1^2} \!:\! y_1 \!:\! z), \\
\eta'_2 \colon &(x_0 \!:\! x_1 \!:\! y_0 \!:\! y_1 \!:\! z) \mapsto (x_0 \!:\! x_1 \!:\! y_0 \!:\! - y_1 - \frac{b_6}{y_0^2} \!:\! z), \\
\theta' \colon &(x_0 \!:\! x_1 \!:\! y_0 \!:\! y_1 \!:\! z) \mapsto (x_0 \!:\! x_1 \!:\! \frac{b_6}{a_6} y_1 \!:\! \frac{a_6}{b_6} y_0 \!:\! z).
\end{split}
\]
Furthermore, we have the following relations
\[
\begin{split}
\eta'_2 &= \theta' \circ \eta'_1 \circ \theta', \\
\eta'_1 \circ \theta' &= \sigma_{12}^{-1} \circ \theta^{-1} \sigma_{21} = \sigma_{11}^{-1} \circ \theta^{-1} \circ \sigma_{22}, \\
\theta' \circ \eta'_1 &= \eta'_2 \circ \theta' = \sigma_{22}^{-1} \circ \theta \circ \sigma_{11} = \sigma_{21}^{-1} \circ \theta \circ \sigma_{12}.
\end{split}
\]
We refer the readers to \cite{Kal} for a general and theoretical treatment of relations of Sarkisov links.

\begin{Thm}
The birational automorphism group $\Bir (X')$ of $X'$ is generated by $\Aut (X')$ and the birational involutions $\eta'_1$, $\theta'$ and $\iota'$.
Moreover, $\theta'$ is biregular if and only if $a_6$ is proportional to $b_6$.
\end{Thm}

\begin{proof}
By the Sarkisov program (see \cite{Corti1}), any birational automorphism $\nu$ of $X'$ is the composite of Sarkisov links $\nu_i \colon V_i \ratmap V_{i+1}$ and an automorphism $\mu$ of $X'$:
\[
\nu \colon X' = V_0 \overset{\nu_0}{\ratmap} V_1 \overset{\nu_1}{\ratmap} \cdots \overset{\nu_{n-1}}{\ratmap} V_n = X' \overset{\mu}{\to} X'.
\]
Note that $V_i \in \{X',X_1,X_2\}$ (or $V_i \in \{X',X\}$) and $\nu_i \in \Sigma$.
Let $k \ge 1$ be the minimum number such that $V_k = X'$.
By considering all the combinations of links $\nu_0,\dots,\nu_{k-1}$, the birational map $X' = V_0 \ratmap V_1 \ratmap \cdots \ratmap V_k = X'$ is one of $\eta'_1$, $\eta'_2 = \theta' \circ \eta'_1 \circ \theta'$, $\theta'$, $\eta'_1 \circ \theta'$, $\eta'_2 \circ \theta' = \theta' \circ \eta'_1$ and $\iota'$.
It follows that $\nu$ is the composite of $\eta'_1$, $\theta'$, $\iota'$ and an automorphism of $X'$.

We prove the remaining part.
It follows immediately from the explicit description of $\theta'$ that if $a_6 \sim b_6$, then $\theta'$ is biregular.
Suppose that $a_6 \not\sim b_6$.
Let $Y_2 \to X_2$ be the divisorial extraction of $X_2$ centered at the $cAx/2$ point $\msp_3$ and $E_2$ its exceptional divisor.
Then, by the proof of Proposition \ref{prop:theta} and Remark \ref{rem:thetadiv}, $E_2$ is not contracted by the induced birational map $Y_2 \ratmap X_1$ and its proper transform on $X_1$ is the divisor $D := (s_0+s_1=0)_{X_1}$. 
Since the link $\sigma_{11}^{-1}$ is centered at $\msp_1 \in X_1$ and $D$ does not pass through $\msp_1$, we see that $D$ cannot be contracted by $\sigma_{11}^{-1}$, and we denote by $D'$ the proper transform of $D$ via $\sigma_{11}^{-1}$.
By the construction of $D'$, it is contracted to $\msp_3 \in X_2$ via $\theta \circ \sigma_{11}$.
Now, by the explicit description of $\sigma_{21}$ and $\sigma_{21}^{-1}$, the link $\sigma_{21}$ induces an isomorphism between open neighborhoods of $\msp_3 \in X_2$ and $\msp'_1 \in X'$.
This shows that $D'$ is contracted to $\msp'_1 \in X'$ via $\theta'$.
Therefore $\theta'$ is not biregular and the proof is completed.
\end{proof}

\begin{Rem}
Assume that $a_6$ is general.
Here, as a generality condition, we require that there is no non-trivial automorphism of $\mbP (1,1,3)$ which leaves $(a_6 = 0)$ invariant.
In this case, we describe $\Aut (X')$ in detail without giving a proof.
If $(a_6,b_6,c_8)$ is asymmetric, then $\Aut (X') = \{\operatorname{id}\}$.
This can be proved by a similar way as in the proof of Proposition \ref{prop:sym}.
We keep the same generality of $a_6$ and consider the symmetric triplet $(a_6,a_6,c_8)$.
In this case, the birational involution $\theta'$ is a biregular automorphism interchanging $y_0$ and $y_1$, and $\Aut (X')$ is generated by $\theta'$. 
In both of the above two cases, $\Bir (X')$ is generated by $\eta'_1$, $\iota'$, $\theta'$ and the only difference is whether $\theta'$ is biregular or not.
Now we fix a general $a_6$ and $c_8$ and let $e_6 \in \mbC [x_0,x_1]$ be a general homogeneous polynomial of degree $6$.
For $t \in \mbC$, let $X'_t$ be the weighted hypersurface corresponding to the triplet $(a_6,a_6 + t e_6,c_8)$.
Then, the above observation shows that $\Bir (X'_t)$ remains the same as a group for $t$ belonging to a small open disk $\Delta \ni 0$ while $\Aut (X'_0) \cong \mbZ/2 \mbZ$ and $\Aut (X'_t) = \{\operatorname{id}\}$ for $t \in \Delta \setminus \{0\}$.
\end{Rem}

\end{document}